\documentclass[11pt,oneside,reqno]{amsart}

\usepackage{amsmath,amssymb,amsthm,textcomp, mathtools}
\usepackage{amsfonts,graphicx}
\usepackage[mathscr]{eucal}
\usepackage{color}
\usepackage{csquotes}

\numberwithin{equation}{section}

\theoremstyle{definition}

\numberwithin{equation}{section}

\newtheorem{theorem}{\bf Theorem}[section]
\newtheorem{remark}{\bf Remark}[section]
\newtheorem{proposition}{Proposition}[section]
\newtheorem{lemma}{Lemma}[section]
\newtheorem{corollary}{Corollary}[section]

\newtheorem{definition}{Definition}[section]
\newtheoremstyle
    {remarkstyle}
    {}
    {11pt}
    {}
    {}
    {\bfseries}
    {:}
    {     }
    {\thmname{#1} \thmnumber{#2} }

\theoremstyle{remarkstyle}

\begin{document}
\title{On Transient Probabilities of Fractional Birth-Death Process}
\author{Kuldeep Kumar Kataria}
\address{K. K. Kataria, Department of Mathematics,
 Indian Institute of Technology Bhilai, Durg, 491002, INDIA.}
 \email{kuldeepk@iitbhilai.ac.in}
 \author{Pradeep Vishwakarma}
 \address{P. Vishwakarma, Department of Mathematics,
 	Indian Institute of Technology Bhilai, Durg, 491002, INDIA.}
 \email{pradeepv@iitbhilai.ac.in}

\subjclass[2010]{Primary: 60J27; Secondary: 60J80}

\keywords{extinction probability; fractional birth-death process; Adomian decomposition method; transient probabilities; cumulative births; Caputo derivative; Riemann-Liouville integral.}
\date{\today}

\maketitle

\begin{abstract}
	We study a fractional birth-death process with state dependent birth and death rates. It is defined using a system of fractional differential equations that generalizes the classical birth-death process introduced by Feller (1939). We obtain the closed form expressions for its transient probabilities using Adomian decomposition method. In this way, we obtain the unknown transient probabilities for the classical birth-death process (see Feller (1968), p. 454). Its various distributional properties are studied. For the case of linear birth and death rates, the obtained results are verified with the existing results. Also, we discuss the cumulative births in the fractional linear birth-death process. Later, we consider a time-changed linear birth-death process where we discuss the asymptotic behaviour of the distribution function of its extinction time at zero.
\end{abstract}

\section{Introduction}
The birth-death process (BDP) is a classical example of the continuous-time and discrete state-space Markov process which is used to model the growth of population over time.  It was introduced by Feller (1939) and has been used in various fields, such as, physics, ecology, biology, queuing theory, \textit{etc}. For instance, in queuing theory,  the notion of states represent the number of people waiting in a queue in which one person can join or leave the queue at any point of time. In BDP, the rates of transition depend on  state of process at any given time $t\ge0$. Let $\lambda_n,\ n\ge0$ and $\mu_n,\ n\ge1$ be the birth and death rates in BDP, respectively. In an infinitesimal time interval of length $h$ such that $o(h)/h\to0$ as $h\to0$, the probability of a birth  is $\lambda_nh+o(h)$, the probability of a death is $\mu_nh+o(h)$, the probability of no change is $1-(\lambda_n+\mu_n)h+o(h)$ and the probability of any other event is $o(h)$. Recently, a generalization of BDP in which at any instance multiple births or deaths can occur with some positive probability has been studied (see Vishwakarma and Kataria (2024a, 2024b)).
 
Let us assume that there is no possibility of immigration in the BDP  $\{N(t)\}_{t\ge0}$, that is, $\lambda_0=0$. If at time $t=0$ there is exactly one progenitor then the system of differential equations that governs its state probabilities
$
    p(n,t)=\mathrm{Pr}\{N(t)=n\}
$
is given by (see Feller (1968), p. 454):
\begin{equation}\label{bdpequ}
    \mathcal{D}_tp(n,t)=
        -(\lambda_n+\mu_n)p(n,t)+\lambda_{n-1}p(n-1,t)+\mu_{n+1}p(n+1,t),\ n\ge0,
\end{equation}
with initial conditions $p(1,0)=1$ and $p(n,t)=0$ for all $n\ne1$, where $\mathcal{D}_t={\mathrm{d}}/{\mathrm{d}t}$ denotes the first order derivative. Here, $\mu_0=0$ and $p(-1,t)=0$ for all $t\ge0$.

It is observed that for any arbitrary choices of birth and death rates there exist a solution $p(n,t)$ of (\ref{bdpequ}) such that $\sum_{n=0}^{\infty}p(n,t)\leq1$ (see Feller (1968)). However, the closed form expression for its solution is not known yet. The detailed study on the behaviour of solution such that $\sum_{n=0}^{\infty}p(n,t)<1$ has been done in Ledermann and Reuter (1954), Karlin and McGregor (1957a, 1957b). Moreover, if the sequences of birth rates $\{\lambda_n\}_{n\ge1}$ and death rates $\{\mu_n\}_{n\ge1}$ increase sufficiently slowly then the solution is unique and satisfy the regularity condition, that is, $\sum_{n=0}^{\infty}p(n,t)=1$. For more details on the regularity of continuous-time Markov processes, we refer the reader to Resnick (2002). 

In this paper, the transition rates $\lambda_n$'s and $\mu_n$'s are chosen such that the BDP is regular. If the transition rates in BDP are linear, that is, $\lambda_n=n\lambda$ and $\mu_n=n\mu$ then it is called the linear BDP. The explicit expressions for the state probabilities of linear BDP are given in Bailey (1964). The fractional variants of linear BDP was introduced and studied by Orsingher and Polito (2011), where they obtained its transient probabilities in three differents cases of birth and death rates. It is observed that the fractional variant of linear BDP has faster mean growth than the linear BDP. Hence, it provides a better mathematical model for a system governed by a BDP under accelerating conditions. Some integrals for the linear case are discussed in Vishwakarma and Kataria (2024c). 

In this paper, we consider a time-fractional variants of the BDP. We denote it by $\{N^\alpha(t)\}_{t\ge0}$, $0<\alpha\leq1$ and call it the fractional birth-death process (FBDP). It is a process whose state probabilities $p^\alpha(n,t)=\mathrm{Pr}\{N^\alpha(t)=n\}$, $n\ge0$ solve the following system of fractional differential equations: 
\begin{equation}\label{diffequ}
	\mathcal{D}^\alpha_tp^\alpha(n,t)=
		-(\lambda_n+\mu_n)p^\alpha(n,t)+\lambda_{n-1}p^\alpha(n-1,t)+\mu_{n+1}p^\alpha(n+1,t),\ n\ge0,
\end{equation}
with initial conditions
\begin{equation}\label{initial}
	p^\alpha(n,0)=\begin{cases}
		1,\ n=1,\\
		0,\ n\ne1,
	\end{cases}
\end{equation}
where the derivative involved in (\ref{diffequ}) is the Caputo fractional derivative defined by (Kilbas \textit{et al.} (2006))
\begin{equation}\label{caputoder}
	\mathcal{D}_t^\alpha f(t)=\begin{cases}
		\frac{1}{\Gamma(1-\alpha)}\int_{0}^{t}\frac{f'(s)\,\mathrm{d}s}{(t-s)^\alpha},\ 0<\alpha<1,\\
		f'(t),\ \alpha=1,
	\end{cases}
\end{equation} 
for a suitable choice of $f(\cdot)$. Its Laplace transform is given by
\begin{equation}\label{frderlap}
	\int_{0}^{\infty}e^{-wt}	\mathcal{D}_t^\alpha f(t)\,\mathrm{d}t=w^{\alpha}\int_{0}^{\infty}e^{-wt}f(t)\,\mathrm{d}t-w^{\alpha-1}f(0),\ w>0.
\end{equation}

To the best of our knowledge, the solution for system of equations given in (\ref{diffequ}) is not known for arbitrary choices of $\lambda_n$'s and $\mu_n$'s, even in the specific case of $\alpha = 1$, where it reduces to the system of equations (\ref{bdpequ}) corresponding to BDP. In this study, the Adomian decomposition method (ADM) is utilize to solve (\ref{diffequ}). We obtain the series expressions for transient probabilities in FBDP. For any process, the transient probabilities and associated moments are important statistical quantities for practical purposes. For example, in an epidemic model, one of the most important quantity of interest is the expected number of individuals infected by a disease up to some finite time, and to determine this, we require these probabilities.

This paper is organized as follows:

In the following section, we collect some definitions and known results.

In Section \ref{sec3}, we define the FBDP whose state probabilities solve a system of fractional differential equations. We establish a characterization for FBDP as a time-changed BDP. We use the ADM to solve the governing system of differential equations for FBDP, and derive a closed form expressions of its state probabilities. Also, some of its other distributional properties are studied in detail.

In Section \ref{sec4}, we obtain additional properties of the FBDP with linear birth and death rates. We call it the linear FBDP which was introduced and studied by Orsingher and Polito (2011). We obtain the series representations for its state probabilities. Later, we consider the cumulative births and study its joint distributional properties with the linear FBDP.

In the last section, we consider a time-changed BDP where the time changes according to an  inverse subordinator. It reduces to the FBDP as a particular case. In the case of linear birth and death rates, we discuss some asymptotic results at zero for the distribution function of its extinction time. 
\section{Preliminaries}\label{pre}
Here, we collect some definitions and results related to subordinator and its inverse. 
\subsection{Subordinator and its first hitting time}\label{subdef}
	A subordinator $\{S^\phi(t)\}_{t\ge0}$ is a non-decreasing L\'evy process with state space $([0,\infty), \mathcal{B}[0,\infty))$ whose Laplace transform is given by (see Applebaum (2004))
	\begin{equation*}
		\mathbb{E}e^{-\eta S^\phi(t)}=e^{-t\phi(\eta)},\ \eta>0.
	\end{equation*} 
	Here, $\phi(\eta)$ is a Bern\v stein function 
	\begin{equation}\label{lapexp}
		\phi(\eta)=\int_{0}^{\infty}(1-e^{-\eta x})\nu(\mathrm{d}x),
	\end{equation}
	where  $\nu(\cdot)$ is the L\'evy measure of $S^\phi(t)$. So,	$
		\nu(-\infty, 0)=0$ and $\int_{0}^{\infty}\min\{x,1\}\nu(\mathrm{d}x)<\infty.
	$
	The function $\phi(\eta)$ is known as the Laplace exponent of the corresponding subordinator.
	In particular, if $\phi(\eta)=\eta^\alpha$, $\alpha\in(0,1)$ then it is called an $\alpha$-stable subordinator, and we denote it by $\{S^\alpha(t)\}_{t\ge0}$.

 The first passage time process $\{E^\phi(t)\}_{t\ge0}$ defined by 
	\begin{equation}\label{invsub}
		E^\phi(t)=\inf\{u\ge0:S^\phi(u)>t\}
	\end{equation}
	is called the inverse subordinator. The first hitting time of an $\alpha$-stable subordinator$\{S^\alpha(t)\}_{t\ge0}$, $\alpha\in(0,1)$ is known as the inverse $\alpha$-stable subordinator, we denote it by $\{E^\alpha(t)\}_{t\ge0}$.	The Laplace transform of its density function is given by (see Meerscheart \textit{et al.} (2011)) 
	\begin{equation}\label{ainvsublap}
		\int_{0}^{\infty}e^{-w t}\mathrm{Pr}\{E^\alpha(t)\in\mathrm{d}x\}\,\mathrm{d}t=w^{\alpha-1}e^{-w^\alpha x}\,\mathrm{d}x,\ w>0.
	\end{equation}
	
	Let $h(x,t)=\mathrm{Pr}\{E^\phi(t)\in\mathrm{d}x\}/\mathrm{d}x$ be the density of inverse subordinator defined in (\ref{invsub}). It solves the following generalized fractional Cauchy problem (see Kolokoltsov (2009)):
	\begin{equation}\label{invsubeq}
		\begin{cases}
			\mathscr{D}_th(x,t)=-\mathcal{D}_xh(x,t),\\
			h(0,t)=\int_{t}^{\infty}\nu(\mathrm{d}x),
		\end{cases}
	\end{equation}
	where $\mathscr{D}_t$ is the generalized fractional differential operator defined as
	\begin{equation}\label{gfrder}
		\mathscr{D}_tf(t)=\int_{0}^{\infty}(f(t)-f(t-x))\nu(\mathrm{d}x)
	\end{equation}
	for a suitable function $f(\cdot)$, and $\nu(\cdot)$ is the L\'evy measure of the associated subordinator.
	
	In particular, for the inverse $\alpha$-stable subordinator $\{E^\alpha(t)\}_{t\ge0}$, $\alpha\in(0,1)$, the operator $\mathscr{D}_t$ reduces to the Caputo fractional derivative as defined in (\ref{caputoder}).

Next, we give a brief introduction of the Adomian decomposition method.
\subsection{Adomian decomposition method (ADM)}
Let $Q$ be a nonlinear operator and $\phi$ be a known function, and let us consider the following functional equation:
 \begin{equation}\label{admdef}
 	u=\phi+Q(u),
 \end{equation}
 where it is assumed that the solution $u$ of (\ref{admdef}) and  term $Q(u)$ can be expressed as absolutely convergent series $u=\sum_{n=0}^{\infty}u_n$ and $Q(u)=\sum_{n=0}^{\infty}A_n(u_0,u_1,\ldots,u_n)$, respectively (see Adomian (1986), (1994)). Thus, (\ref{admdef}) can be rewritten as follows:
 \begin{equation*}
 \sum_{n=0}^{\infty}u_n=\phi + \sum_{n=0}^{\infty}A_n(u_0,u_1,\ldots ,u_n).
 \end{equation*}
 Here, $A_n$ is known as the $n$th Adomian polynomial in terms of $u_0,u_1,\ldots,u_n$.
 In ADM, the series components $u_n$'s are obtained by using the following recursive relation:
 \begin{equation*}
 	u_0=\phi\ \ \ \ \mathrm{and}\ \ \ \ u_n=A_{n-1}(u_0,u_1,\ldots,u_{n-1}).
 \end{equation*}
 
 The key aspect of ADM is the calculation of Adomian polynomials. Adomian (1986) introduced a technique to derive these polynomials by parametrizing the solution $u$. Moreover, when $Q$ is a linear operator, such as $Q(u) = u$, the polynomial $A_n$ simplifies to $u_n$. For further details on these polynomials, we refer the reader to Rach (1984), Duan (2010, 2011), Kataria and Vellaisamy (2016).

  It is important to note that there are no nonlinear terms involved in the system of differential equations corresponding to the FBDP. So, the ADM can be effectively use to solve the system of equations (\ref{diffequ}) to obtain its state probabilities.

\section{Fractional birth-death process}\label{sec3}
Orsingher and Polito (2011) studied the FBDP for linear birth and death rates, that is, $\lambda_n=n\lambda$ and $\mu_n=n\mu$ for all $n\ge0$, where $\lambda$ and $\mu$ are positive constants. Its transient probabilities are obtained using the state probabilities of linear BDP and a time-changed  representation of the linear FBDP in terms of the linear BDP. It is established using the governing differential equation of the probability generating function (pgf) of linear FBDP.

First, we show that a similar time-changed relationship holds for the FBDP. However, it is not possible to find the governing equation for the pgf of FBDP. So, the method used by Orsingher and Polito (2011) is not applicable in this case.  We use a different method to established a time-changed representation of FBDP. Also, the method used is applicable for the linear birth and death rates too. 

\begin{theorem}\label{thmsub}
	Let $\{T^{2\alpha}(t)\}_{t\ge0}$, $0<\alpha\leq1$ be a random process independent of the BDP $\{N(t)\}_{t\ge0}$, whose density function $\mathrm{Pr}\{T^{2\alpha}(t)\in\mathrm{d}x\}/\mathrm{d}x$ is the folded solution of the following fractional Couchy problem:
	\begin{equation*}
		\mathcal{D}_t^{2\alpha}h^\alpha(x,t)=\mathcal{D}^2_xh^\alpha(x,t),\ 0<\alpha\leq1,\,x\in\mathbb{R},\,t\ge0,
	\end{equation*}
	with initial condition $h^\alpha(x,0)=\delta(x)$ for $0<\alpha\leq1$, and also $h^\alpha_t(x,0)=0$ when $1/2<\alpha\leq1$. Here, $\delta(\cdot)$ is the Dirac delta function and $\mathcal{D}_x^2$ denotes the second order derivative with respect to $x$. Then, FBDP has the following characterization:
	\begin{equation}\label{subr}
		N^\alpha(t)\overset{d}{=}N(T^{2\alpha}(t)),\ t\ge0,\ 0<\alpha\leq1,
	\end{equation}
	where $\overset{d}{=}$ denotes the equality in distribution.
\end{theorem}
\begin{proof}
	It is sufficient to show that $\mathrm{Pr}\{N(T^{2\alpha}(t))=n\}$, $n\ge0$ solves (\ref{diffequ}). Note that
	\begin{equation*}
		\mathrm{Pr}\{N(T^{2\alpha}(t))=n\}=\int_{0}^{\infty}\mathrm{Pr}\{N(x)=n\}\mathrm{Pr}\{T^{2\alpha}(t)\in\mathrm{d}x\},\ n\ge0
	\end{equation*}
	and its Laplace transform is given by
	\begin{equation}\label{fbdppmflap}
		\int_{0}^{\infty}e^{-wt}\mathrm{Pr}\{N(T^{2\alpha}(t))=n\}\,\mathrm{d}t=w^{\alpha-1}\int_{0}^{\infty}p(n,x)e^{-w^\alpha x}\,\mathrm{d}x,\ n\ge0,\ w>0.
	\end{equation}
	Here, we have used the following result (see Orsingher and Beghin (2004), Eq. (3.3) for $c=1$):
	\begin{equation}\label{diffslap}
		\int_{0}^{\infty}e^{-wt}\mathrm{Pr}\{T^{2\alpha}(t)\in\mathrm{d}x\}\,\mathrm{d}t=w^{\alpha-1}e^{-w^\alpha x}\,\mathrm{d}x,\ w>0.
	\end{equation}	
	On taking the Laplace transform on both sides of (\ref{bdpequ}), we have
	\begin{equation*}
		w\int_{0}^{\infty}e^{-wt}p(n,t)\,\mathrm{d}t-p(n,0)=\int_{0}^{\infty}e^{-wt}(-(\lambda_n+\mu_n)p(n,t)+\lambda_{n-1}p(n-1,t)+\mu_{n+1}p(n+1,t))\,\mathrm{d}t.
	\end{equation*}
	So,
	\begin{multline}\label{bdppmflap}
		w^\alpha w^{\alpha-1}\int_{0}^{\infty}e^{-w^\alpha t}p(n,t)\,\mathrm{d}t-p(n,0)w^{\alpha-1}\\=w^{\alpha-1}\int_{0}^{\infty}e^{-w^\alpha t}(-(\lambda_n+\mu_n)p(n,t)+\lambda_{n-1}p(n-1,t)+\mu_{n+1}p(n+1,t))\,\mathrm{d}t.
	\end{multline}
	On substituting (\ref{fbdppmflap}) in (\ref{bdppmflap}), we get
	\begin{align}\label{pf1}
		w^\alpha\int_{0}^{\infty}&e^{-wt}\mathrm{Pr}\{N(T^{2\alpha}(t))=n\}\,\mathrm{d}t-\mathrm{Pr}\{N(T^{2\alpha}(0))=n\}w^{\alpha-1}\nonumber\\
		&=\int_{0}^{\infty}e^{-w t}(-(\lambda_n+\mu_n)\mathrm{Pr}\{N(T^{2\alpha}(t))=n\}\nonumber\\
		&\hspace{1.5cm} +\lambda_{n-1}\mathrm{Pr}\{N(T^{2\alpha}(t))=n-1\}+\mu_{n+1}\mathrm{Pr}\{N(T^{2\alpha}(t))=n+1\})\,\mathrm{d}t,\, w>0,
	\end{align}
	where we have used $p(n,0)=\mathrm{Pr}\{N(T^{2\alpha}(0))=n\}$. Now, on taking the inverse Laplace transform of (\ref{pf1}) and using (\ref{frderlap}), we get 
	\begin{align*}
		\mathcal{D}_t^\alpha\mathrm{Pr}\{N(T^{2\alpha}(t))=n\}&=-(\lambda_n+\mu_n)\mathrm{Pr}\{N(T^{2\alpha}(t))=n\}\\
		&\ \ +\lambda_{n-1}\mathrm{Pr}\{N(T^{2\alpha}(t))=n-1\}+\mu_{n+1}\mathrm{Pr}\{N(T^{2\alpha}(t))=n+1\},\ n\ge0.
	\end{align*}
	The proof follows from the uniqueness of distribution.
\end{proof}
\begin{remark}
	The distribution of $T^{2\alpha}(t)$ can be explicitly derived for some specific values of $\alpha$. For instance, if $\alpha=2^{-l}$, $l\ge1$ then the density of $T^{2\alpha}(t)$ coincides with the density of $(l-1)$th iterated Brownian motion, that is,
	$\mathrm{Pr}\{T^{1/2^{l-1}}(t)\in\mathrm{d}x\}=\mathrm{Pr}\{|B_1(|B_2(\dots|B_l(t)|\dots)|)|\in\mathrm{d}x\}$, where $B_i(t)$'s are independent Brownian motions (see Orsingher and Beghin (2009), Theorem 2.1). In particular, for $l=1$, its density coincides with that of a reflecting Brownian motion $\{|B(t)|\}_{t\ge0}$. Thus, $\{N^{1/2}(t)\}_{t\ge0}$ is a BDP at reflecting Brownian time.
\end{remark}
\begin{remark}\label{re2.2}
	Let $\{S^\alpha(t)\}_{t\ge0}$, $0<\alpha<1$ be an $\alpha$-stable subordinator (for definition see Section \ref{subdef}) and $\{E^\alpha(t)\}_{t\ge0}$ be its first hitting time  as defined in (\ref{invsub}). Then, from Theorem 3.1 of Meerschaert \textit{et al.} (2011), we have $E^\alpha(t)\overset{d}{=}T^{2\alpha}(t)$. It follows that $N^\alpha(t)\overset{d}{=}N(E^\alpha(t))$, where $\{E^\alpha(t)\}_{t\ge0}$ is independent of $\{N(t)\}_{t\ge0}$. 
\end{remark}

\subsection{Application of ADM to FBDP} Here, we apply the ADM to solve the system of differential equations given in (\ref{diffequ}). First, we recall the definition of Riemann-Liouville fractional integral.
\begin{definition}
	Let $g(\cdot)$ be an real valued integrable function. Then, its Riemann-Liouville fractional integral is defined as (see Kilbas \textit{et al.} (2006), p. 79)
	\begin{equation*}
		I_t^\alpha g(t)\coloneqq\frac{1}{\Gamma(\alpha)}\int_{0}^{t}\frac{g(s)}{(t-s)^{1-\alpha}}\,\mathrm{d}s,\ \alpha>0,
	\end{equation*}
	where $I^\alpha_t$ is known as the Riemann-Liouville fractional integral operator of order $\alpha>0$.
\end{definition}

Further, for $\beta>0$, we have the following result (see Kilbas \textit{et al.} (2006), p. 81):
\begin{equation}\label{frintresult}
I_t^\alpha(t^{\beta-1})=\frac{\Gamma(\beta)t^{\alpha+\beta-1}}{\Gamma(\alpha+\beta)},\ \alpha>0.
\end{equation}

Note that the Riemann-Liouville fractional integral operator $I_t^\alpha$ is linear. So, the Adomian polynomials $A_k$'s for $Q(u(t))=I_t^\alpha (u(t))$ are given by $A_k(u_0(t), u_1(t)$, $\ldots,u_k(t))=I_t^\alpha(u_k(t))$
for all $k\ge0$.

 On applying $I_t^\alpha$ on both sides of $(\ref{diffequ})$ and substituting $p^\alpha(n,t)=\sum_{k=0}^{\infty}p_k^\alpha(n,t)$ for all $n\ge0$, we get
\begin{equation*}
    \sum_{k=0}^{\infty}p_k^\alpha(n,t)=
        p^\alpha(n,0)+\sum_{k=0}^{\infty}I_t^\alpha(-(\lambda_n+\mu_n)p_k^\alpha(n,t)+\lambda_{n-1}p_k^\alpha(n-1,t)+\mu_{n+1}p_k^\alpha(n+1,t)),\ n\ge0.
\end{equation*}
Using ADM, we have

\begin{equation}\label{adm2}
	p_k^\alpha(n,t)=
		I_t^\alpha(-(\lambda_n+\mu_n)p_{k-1}^\alpha(n,t)+\lambda_{n-1}p_{k-1}^\alpha(n-1,t)+\mu_{n+1}p_{k-1}^\alpha(n+1,t)),\ n\ge0
\end{equation}
and
\begin{equation}\label{adm1}
	p_0^\alpha(n,t)=p^\alpha(n,0)=\begin{cases}
		1,\ n=1,\\
		0,\ n\ne 1,
	\end{cases}
\end{equation}
where $p_k^\alpha(-1,t)=0$ for all $t\ge0$.

 Next result provides a sufficient condition under which the series components of the state probabilities of FBDP are zero.
\begin{proposition}\label{prop1}
For $t\ge0$ and $n\ge k+2$, the series component  $p_k^\alpha(n,t)$ defined by (\ref{adm2}) vanishes, that is, $p_k^\alpha(n,t)=0$ given $n-k\ge2$.
\end{proposition}
\begin{proof}
    It is sufficient to show that
    \begin{equation}\label{3.5}
        p_k^\alpha(n+1+k,t)=0, \ n\ge1,\ k\ge1.
    \end{equation}
    On taking $n=k=1$ in (\ref{adm2}) and using $(\ref{adm1})$, we get
    \begin{equation*}
        p_1^\alpha(3,t)=I_t(-(\lambda_{3}+\mu_{3})p_0^\alpha(3,t)+\lambda_{2}p_0^\alpha(2,t)+\mu_{4}p_0^\alpha(4,t))=I_t0=0.
    \end{equation*}
    So, $(\ref{3.5})$ holds for $n=k=1$.  For $k=1$ and $n\ge1$, from $(\ref{adm2})$, we get
    \begin{equation*}
        p_1^\alpha(n+2,t)=I_t^\alpha(-(\lambda_{n+2}+\mu_{n+2})p_0^\alpha(n+2,t)+\lambda_{n+1}p_0^\alpha(n+1,t)+\mu_{n+3}p_0^\alpha(n+3,t))=I_t^\alpha 0=0,
    \end{equation*}
    where we have used (\ref{adm1}). Hence, the result holds for $k=1$ and $n\ge1$.
    
     Suppose the result holds for some $k=m\ge2$ and for all $n\ge1$, that is,
    \begin{equation}\label{3.6}
        p_l^\alpha(n+1+m,t)=0,\ n\ge1.
    \end{equation}
     Then, for $k=m+1$, from (\ref{adm2}) and using (\ref{3.6}), we get
    \begin{align*}
        p_{m+1}^\alpha(n+m+2,t)&=I_t^\alpha(-(\lambda_{n+m+2}+\mu_{n+m+2})p_l^\alpha(n+m+2,t)\\
        &\ \ +\lambda_{n+m+1}p_l^\alpha(n+m+1,t)+\mu_{n+m+3}p_l^\alpha(n+m+3,t))=I_t^\alpha0=0.
    \end{align*}
   The proof is complete using the method of induction.
\end{proof}

 	Now, we explicitly compute some  of the series components. 	
 	 Let $\Lambda_n=\lambda_n+\mu_n$ and $\theta_n=\lambda_n\mu_{n+1}$ for all $n\ge1$. Then, from (\ref{adm2}), we have
	\begin{equation}\label{ncons1}
		p_k^\alpha(n,t)=\begin{cases}
			I_t^\alpha(\mu_1 p_{k-1}^\alpha(1,t)),\ n=0,\  k\ge1,\vspace{0.1cm}\\
			I_t^\alpha(-\Lambda_1p_{k-1}^\alpha(1,t)+\mu_2p_{k-1}^\alpha(2,t)),\ n=1,\ k\ge1,\vspace{0.1cm}\\
			I_t^\alpha(-\Lambda_np_{k-1}^\alpha(n,t)+\lambda_{n-1}p_{k-1}^\alpha(n-1,t)+\mu_{n+1} p_{k-1}^\alpha(n+1,t)),\ n\ge2,\ k\ge1.
		\end{cases}
	\end{equation}
On taking $n=0$ in (\ref{ncons1}), we have
$p_k^\alpha(0,t)=I_t^\alpha(\mu_1 p_{k-1}^\alpha(1,t)),\ k\ge1.$
So, by using (\ref{frintresult}), we get
\begin{equation}\label{0.1}
	p_1^\alpha(0,t)=I_t^\alpha(\mu_1 p_0^\alpha(1,t))=\mu_1 I_t^\alpha(t^0)=\frac{\mu_1 t^\alpha}{\Gamma(\alpha+1)}.
\end{equation}
For $n=1$, we have
\begin{equation}\label{1.1}
	p_1^\alpha(1,t)=I_t^\alpha(-\Lambda_1p_{0}^\alpha(1,t)+\mu_{2} p_{0}^\alpha(2,t)),\ k\ge1.
\end{equation}
In view of Proposition \ref{prop1}, using (\ref{frintresult}), we get $p_1^\alpha(1,t)=-\Lambda_1 t^\alpha/\Gamma(\alpha+1)$. So, $p_2^\alpha(0,t)=-\mu_1\Lambda_1t^{2\alpha}/\Gamma(2\alpha+1)$.

Now, on taking $n=2$ in (\ref{ncons1}) and using Proposition \ref{prop1}, we get
\begin{equation*}
	p_1^\alpha(2,t)=I_t^\alpha(-\Lambda_2p_0^\alpha(2,t)+\lambda_1p_0^\alpha(1,t)+\mu_3p_0^\alpha(3,t))=\frac{\lambda_1t^\alpha}{\Gamma(\alpha+1)}.
\end{equation*}
Thus, 
\begin{equation*}
	p_2^\alpha(1,t)=I_t^\alpha(-\Lambda_1p_1^\alpha(1,t)+\mu_2p_1^\alpha(2,t))=(\Lambda_1^2+\theta_1)\frac{t^{2\alpha}}{\Gamma(2\alpha+1)}.
\end{equation*}
Similarly, we get the following series components:
	\begin{align*}		   p_3^\alpha(0,t)&=I_t^\alpha(\mu_1p_2^\alpha(1,t))=\mu_1(\Lambda_1^2+\theta_1)\frac{t^{3\alpha}}{\Gamma(3\alpha+1)},\\
		p_2^\alpha(2,t)&=I_t^\alpha(-\Lambda_2p_1^\alpha(2,t)+\lambda_1p_1^\alpha(1,t)+\mu_3p_1^\alpha(3,t))=-\lambda_1(\Lambda_1+\Lambda_2)\frac{t^{2\alpha}}{\Gamma(2\alpha+1)},\\
		p_3^\alpha(1,t)&=I_t^\alpha(-\Lambda_1p_2^\alpha(1,t)+\mu_2p_2^\alpha(2,t))=-(\Lambda_1^3+2\theta_1\Lambda_1+\theta_1\Lambda_2)\frac{t^{3\alpha}}{\Gamma(3\alpha+1)},\\
		p_4^\alpha(0,t)&=I_t^\alpha(\mu_1p_3^\alpha(1,t))=-\mu_1(\Lambda_1^3+2\theta_1\Lambda_1+\theta_1\Lambda_2)\frac{t^{4\alpha}}{\Gamma(4\alpha+1)},\\
		p_2^\alpha(3,t)&=I_t^\alpha(-\Lambda_3p_1^\alpha(3,t)+\lambda_2p_1^\alpha(2,t)+\mu_4p_1^\alpha(4,t))=\lambda_1\lambda_2\frac{t^{2\alpha}}{\Gamma(2\alpha+1)},\\
		p_3^\alpha(2,t)&=I_t^\alpha(-\Lambda_2p_2^\alpha(2,t)+\lambda_1p_2^\alpha(1,t)+\mu_3p_2^\alpha(3,t))=\lambda_1(\Lambda_1^2+\Lambda_1\Lambda_2+\Lambda_2^2+\theta_1+\theta_2)\frac{t^{3\alpha}}{\Gamma(3\alpha+1)},\\
		p_4^\alpha(1,t)&=I_t^\alpha(-\Lambda_1p_3^\alpha(1,t)+\mu_2p_3^\alpha(2,t))=(\Lambda_1^4+3\theta_1\Lambda_1^2+2\theta_1\Lambda_1\Lambda_2+\theta_1\Lambda_2^2+\theta_1^2+\theta_1\theta_2)\frac{t^{4\alpha}}{\Gamma(4\alpha+1)},\\
		p_5^\alpha(0,t)&=I_t^\alpha(\mu_1p_4^\alpha(1,t))=\mu_1(\Lambda_1^4+3\theta_1\Lambda_1^2+2\theta_1\Lambda_1\Lambda_2+\theta_1\Lambda_2^2+\theta_1^2+\theta_1\theta_2)\frac{t^{5\alpha}}{\Gamma(5\alpha+1)},\\
		p_3^\alpha(3,t)&=I_t^\alpha(-\Lambda_3p_2^\alpha(3,t)+\lambda_2p_2^\alpha(2,t)+\mu_4p_2^\alpha(4,t))=-\lambda_1\lambda_2(\Lambda_1+\Lambda_2+\Lambda_3)\frac{t^{3\alpha}}{\Gamma(3\alpha+1)}.
	\end{align*}

Note that all the series components can be computed recursively. Next, we derive a closed form expression for the $k$th series component.
\begin{theorem}\label{cmpts}
Let $\{a_{n,\,k}\}_{n\ge0,\,k\ge1}$ and $\{b_{n,\,k}\}_{n\ge0,\,k\ge1}$ be two double index sequences of positive integer defined as follows:
\begin{equation}\label{ank}
	a_{n,\,k}=\left[\frac{n+k}{2}\right],\ n\ge0,\,k\ge1
\end{equation} 
and
\begin{equation}\label{bnk}
	b_{n,\,k}=\begin{cases}
		\left[\frac{n+k+1}{2}\right],\ n-k\ne1,\vspace{0.1cm}\\
		0,\ n-k=1,
	\end{cases}
\end{equation}
such that $a_{n,\,k}\leq b_{n,\,k}$ when $n\neq k+1$ for all $n\ge0$ and $k\ge1$, where $[\cdot]$ denotes the greatest integer function. Also, let $\psi_{n,\,k}\big((\lambda_j)_{1}^{a_{n,\,k}},(\mu_j)_{1}^{b_{n,\,k}}\big)$, $n\ge0$, $k\ge0$ be functions of $\lambda_1,\lambda_{2},\dots,\lambda_{a_{n,\,k}}, \mu_1,\mu_{2},\dots,\mu_{b_{n,\,k}}$ such that $\psi_{0,0}=0$, $\psi_{1,0}=1$, $\psi_{0,1}=\mu_1$ and
{\scriptsize\begin{equation}\label{psidef}
	\psi_{n,\,k+1}\big((\lambda_j)_{1}^{a_{n,\,k+1}},(\mu_j)_{1}^{b_{n,\,k+1}}\big)=\begin{cases}
		0,\ n> k+2,\vspace{0.2cm}\\
		\mu_1\psi_{1,\,k}\big((\lambda_j)_{1}^{a_{1,k}},(\mu_j)_{1}^{b_{1,k}}\big),\ n=0,\,k\ge1,\vspace{0.2cm}\\
		\Lambda_1\psi_{1,\,k}\big((\lambda_j)_{1}^{a_{1,k}},(\mu_j)_{1}^{b_{1,k}}\big)+\mu_2\psi_{2,\,k}\big((\lambda_j)_{1}^{a_{2,k}},(\mu_j)_{1}^{b_{2,k}}\big),\ n=1,\,k\ge0,\vspace{0.2cm}\\
		\Lambda_n\psi_{n,\,k}\big((\lambda_j)_{1}^{a_{n,\,k}},(\mu_j)_{1}^{b_{n,\,k}}\big)+\lambda_{n-1}\psi_{n-1,\,k}\big((\lambda_j)_{1}^{a_{n-1,\,k}},(\mu_j)_{1}^{b_{n-1,\,k}}\big)\\
		\hspace{1.5cm}+\mu_{n+1}\psi_{n+1,\,k}\big((\lambda_j)_{1}^{a_{n+1,\,k}},(\mu_j)_{1}^{b_{n+1,\,k}}\big),\ n\ge2,\,k> n-2,\vspace{0.2cm}\\
		\lambda_1\lambda_{2}\cdots\lambda_{n-1},\ n\ge2,\,n-k=2.
	\end{cases}
\end{equation}}
 Then, for all $t\ge0$, the series components $p_k^\alpha(n,t)$, $k\ge0$ of the state probabilities $p^\alpha(n,t)$, $n\ge0$ are given by
\begin{equation}\label{secomps}
	p_k^\alpha(n,t)=\begin{cases}
		(-1)^{k+1}
		\psi_{0,\,k}\big((\lambda_j)_{1}^{a_{0,k}},(\mu_j)_{1}^{b_{0,k}}\big)\frac{t^{k\alpha}}{\Gamma(k\alpha+1)},\ n=0,\, k\ge1,\vspace{0.2cm}\\
		(-1)^{k-n+1}
		\psi_{n,\,k}\big((\lambda_j)_{1}^{a_{n,\,k}},(\mu_j)_{1}^{b_{n,\,k}}\big)\frac{t^{k\alpha}}{\Gamma(k\alpha+1)},\ n\ge1,\, n-k<1,\vspace{0.2cm}\\
		\prod_{j=1}^{n-1}\lambda_j\frac{t^{k\alpha}}{\Gamma(k\alpha+1)},\ n\ge2,\, n-k=1,\vspace{0.2cm}\\
		0,\ n=k=0\ \text{or}\ n-k>1.
	\end{cases}
\end{equation}
\end{theorem}
\begin{proof}
Note that $p_0^\alpha(0,t)=0$ and $p^\alpha_k(n,t)=0$ for $n> k+1$ follow from (\ref{adm1}) and Proposition \ref{prop1}, respectively. Also, for $n=0$, $k=1$ and  $n=k=1$, the result follows from (\ref{0.1}) and (\ref{1.1}), respectively.

Now, let us assume that the result holds for $n\ge0$ and some $k=m\ge2$, that is,
\begin{equation}\label{indhypo}
	p_m^\alpha(n,t)=\begin{cases}
		(-1)^{m+1}
		\psi_{0,m}\big((\lambda_j)_{1}^{a_{0,m}},(\mu_j)_{1}^{b_{0,m}}\big)\frac{t^{m\alpha}}{\Gamma(m\alpha+1)},\ n=0,\vspace{0.2cm}\\
		(-1)^{m-n+1}
		\psi_{n,m}\big((\lambda_j)_{1}^{a_{n,m}},(\mu_j)_{1}^{b_{n,m}}\big)\frac{t^{m\alpha}}{\Gamma(m\alpha+1)},\ n\ge1,\, n-m<1,\vspace{0.2cm}\\
		\prod_{j=1}^{n-1}\lambda_j\frac{t^{m\alpha}}{\Gamma(m\alpha+1)},\ n\ge1,\,n-m=1,\vspace{0.2cm}\\
		0,\ n-m>1.
	\end{cases}
\end{equation}
Then, we have the following three cases:\\
\textit{Case I.} If $n=0$ then from (\ref{ncons1}), we have $	p^\alpha_{m+1}(0,t)=I_t^\alpha(\mu_1p^\alpha_m(1,t))$. So, using induction hypothesis (\ref{indhypo}), we get
\begin{align*}
	p^\alpha_{m+1}(0,t)&=\mu_1I_t^\alpha\left((-1)^{m}
	\psi_{1,m}\big((\lambda_j)_{1}^{a_{1,m}},(\mu_j)_{1}^{b_{1,m}}\big)\frac{t^{m\alpha}}{\Gamma(m\alpha+1)}\right)\\
	&=(-1)^{m}\mu_1\psi_{1,m}\big((\lambda_j)_{1}^{a_{1,m}},(\mu_j)_{1}^{b_{1,m}}\big)\frac{t^{(m+1)\alpha}}{\Gamma((m+1)\alpha+1)},
\end{align*}
where have used (\ref{frintresult}) to get the last equality. So, the result holds for $n=0$ and $k=m+1$ using (\ref{psidef}).\\
\textit{Case II.} Let $n=1$. Then, for $k=m+1$, from (\ref{ncons1}), we have
\begin{align*}
	p^\alpha_{m+1}(1,t)&=I_t^\alpha(-\Lambda_1 p^\alpha_m(1,t)+\mu_{2}p^\alpha_m(2,t))\\
	&=\big(\Lambda_1(-1)^{m+1}\mu_1\psi_{1,m}\big((\lambda_j)_{1}^{a_{1,m}},(\mu_j)_{1}^{b_{1,m}}\big)\\
	&\ \ +\mu_{2}(-1)^{m-1}
	\psi_{2,m}\big((\lambda_j)_{1}^{a_{2,m}},(\mu_j)_{1}^{b_{2,m}}\big)\big)I_t^\alpha\bigg(\frac{t^{m\alpha}}{\Gamma(m\alpha+1)}\bigg)\\
	&=(-1)^{m+1}\psi_{1,m+1}\big((\lambda_j)_{1}^{a_{1,m+1}},(\mu_j)_{1}^{b_{1,m+1}}\big)\frac{t^{(m+1)\alpha}}{\Gamma((m+1)\alpha+1)},
\end{align*}
where we have used (\ref{indhypo}) to get the first equality, and the last step follows from (\ref{frintresult}) and (\ref{psidef}). Hence, the result holds for $n=1$ and $k=m+1$.\\
\textit{Case III.} If $n\ge2$ then in view of Proposition \ref{prop1}, we have $p^\alpha_{m+1}(n,t)=0$ for all $n>m+2$.

Suppose $n-m=2$. Then, from Proposition \ref{prop1} it follows that $p^\alpha_m(n+1,t)=p^\alpha_m(n,t)=0$ for all $t\ge0$, and we have
\begin{align*}
	p^\alpha_{m+1}(m+2,t)&=I_t^\alpha(\lambda_{m+1}p^\alpha_m(m+1,t))\\
	&=I_t^\alpha\left(\lambda_{m+1}
	\psi_{m+1,m}\big((\lambda_j)_{1}^{a_{m+1,m}},(\mu_j)_{1}^{b_{m+1,m}}\big)\frac{t^{m\alpha}}{\Gamma(m\alpha+1)}\right)\\
	&=\lambda_{m+1}
	\big(\Lambda_{m+1}\psi_{m+1,m-1}\big((\lambda_j)_{1}^{a_{m+1,m-1}},(\mu_j)_{1}^{b_{m+1,m-1}}\big)\\
	&\ \ +\lambda_{m}\psi_{m,m-1}\big((\lambda_j)_{1}^{a_{m,m-1}},(\mu_j)_{1}^{b_{m,m-1}}\big)\\
	&\ \ +\mu_{m+2}\psi_{m+2,m-1}\big((\lambda_j)_{1}^{a_{m+2,m-1}},(\mu_j)_{1}^{b_{m+2,m-1}}\big)\big)\frac{t^{(m+1)\alpha}}{\Gamma((m+1)\alpha+1)}\\
	&=\lambda_1\lambda_{2}\cdots\lambda_{m+1}\frac{t^{(m+1)\alpha}}{\Gamma((m+1)\alpha+1)},
\end{align*}
where we have used  $\psi_{m,m-1}\big((\lambda_j)_{1}^{a_{m,m-1}},(\mu_j)_{1}^{b_{m,m-1}}\big)=\lambda_1\lambda_2\cdots\lambda_{m-1}$, $\psi_{m+1,m-1}=0$ and $\psi_{m+2,m-1}$ $=0$.

If $n-m=1$ then $p^\alpha_m(m+2,t)=0$ for all $t\ge0$. So,
\begin{align*}
	p^\alpha_{m+1}(m+1,t)&=I_t^\alpha(-\Lambda_{m+1}p^\alpha_m(m+1,t)+\lambda_{m}p^\alpha_m(m,t))\\
	&=\big(-\Lambda_{m+1}\psi_{m+1,m}\big((\lambda_j)_{1}^{a_{m+1,m}},(\mu_j)_{1}^{b_{m+1,m}}\big)\\
	&\ \ -\lambda_m\psi_{m,m}\big((\lambda_j)_{1}^{a_{m,m}},(\mu_j)_{1}^{b_{m,m}}\big)\big)\frac{t^{(m+1)\alpha}}{\Gamma((m+1)\alpha+1)}\\	
	&=-\psi_{m+1,m+1}\big((\lambda_j)_{1}^{a_{m+1,m+1}},(\mu_j)_{1}^{b_{m+1,m+1}}\big)\frac{t^{(m+1)\alpha}}{\Gamma((m+1)\alpha+1)},
\end{align*}
where the last equality follows from (\ref{psidef}) on using the fact that $\psi_{m+2,m}=0$.

Let $n-m\leq0$. Then, from (\ref{ncons1}), we have
\begin{align*}
	p^\alpha_{m+1}(n,t)&=I_t^\alpha(-\Lambda_np_{m}^\alpha(n,t)+\lambda_{n-1}p_{m}^\alpha(n-1,t)+\mu_{n+1} p_{m}^\alpha(n+1,t))\\
	&=(-1)^{m-n+2}\big(\Lambda_n\psi_{n,m}\big((\lambda_j)_{1}^{a_{n,m}},(\mu_j)_{1}^{b_{n,m}}\big)+\lambda_{n-1}\psi_{n-1,m}\big((\lambda_j)_{1}^{a_{n-1,m}},(\mu_j)_{1}^{b_{n-1,m}}\big)\\
	&\ \ +\mu_{n+1}\psi_{n+1,m}\big((\lambda_j)_{1}^{a_{n+1,m}},(\mu_j)_{1}^{b_{n+1,m}}\big)\big)\frac{t^{(m+1)\alpha}}{\Gamma((m+1)\alpha+1)}\\
	&=(-1)^{m-n+2}\psi_{n,m+1}\big((\lambda_j)_{1}^{a_{n,m+1}},(\mu_j)_{1}^{b_{n,m+1}}\big)\frac{t^{(m+1)\alpha}}{\Gamma((m+1)\alpha+1)},
\end{align*}
where we have used (\ref{indhypo}) to get the second equality and the last step follows from (\ref{psidef}). Thus, the result holds for $n\ge2$ and $k=m+1$. The proof is complete using the method of induction.
\end{proof}
\begin{proposition}
	The functions $\psi_{n,\,k}\big((\lambda_j)_{1}^{a_{n,\,k}},(\mu_j)_{1}^{b_{n,\,k}}\big)$, $n\ge0$, $k\ge0$ as defined in (\ref{psidef}) are homogeneous polynomials of order $k$ in $\lambda_1,\lambda_{2},\dots,\lambda_{a_{n,\,k}},\mu_1,\mu_{2},\dots,\mu_{b_{n,\,k}}$, where $a_{n,\,k}$ and $b_{n,\,k}$ are defined in (\ref{ank}) and (\ref{bnk}), respectively.	
\end{proposition}
\begin{proof}
	Note that $\psi_{n,\,k}=0$ whenever $n-k\ge2$ which is a homogeneous polynomial of any order. Also, $\psi_{0,0}=0$, $\psi_{0,1}=\mu_1$ and $\psi_{1,0}=1$ are homogeneous polynomials. For $n=k=1$, from (\ref{psidef}), we have $\psi_{1,1}(\lambda_1,\mu_1)=(\lambda_1+\mu_1)\psi_{1,0}=(\lambda+\mu)$. So, the result holds for $n=k=1$.
	
	Let us assume that the result holds for all $n\ge0$ and for some $k=m\ge2$, that is, the polynomial $\psi_{n,m}\big((\lambda_j)_{1}^{a_{n,m}},(\mu_j)_{1}^{b_{n,m}}\big)$ is homogeneous of order $m$ for all $n\ge0$. So, from (\ref{psidef}) and using induction hypothesis, it follows that $\psi_{n,m+1}\big((\lambda_j)_{1}^{a_{n,m+1}},(\mu_j)_{1}^{b_{n,m+1}}\big)$ is a homogeneous polynomial of order $m+1$. This completes the proof. 	
\end{proof}

On summing (\ref{secomps}) over the range of $k$, we get the following result:
\begin{theorem}\label{thmpmf}
	For $n\ge1$, the state probabilities of FBDP are given by
	\begin{equation}\label{pn}
		p^\alpha(n,t)=\frac{t^{(n-1)\alpha}}{\Gamma((n-1)\alpha+1)}\prod_{j=1}^{n-1}\lambda_j+\sum_{k=n}^{\infty}(-1)^{k-n+1}
		\psi_{n,\,k}\big((\lambda_j)_{1}^{a_{n,\,k}},(\mu_j)_{1}^{b_{n,\,k}}\big)\frac{t^{k\alpha}}{\Gamma(k\alpha+1)},\ n\ge1
	\end{equation}
	and its extinction probability is 
	\begin{equation}\label{p0}
		p^\alpha(0,t)=\sum_{k=1}^{\infty}(-1)^{k+1}
		\psi_{0,\,k}\big((\lambda_j)_{1}^{a_{0,k}},(\mu_j)_{1}^{b_{0,k}}\big)\frac{t^{k\alpha}}{\Gamma(k\alpha+1)}.
	\end{equation}
\end{theorem}

\begin{remark}
	For $\alpha=1$, the FBDP reduces to the BDP. Its state probabilities are
	\begin{equation}\label{bdppn}
		p(n,t)=\frac{t^{n-1}}{(n-1)!}\prod_{j=1}^{n-1}\lambda_j+\sum_{k=n}^{\infty}(-1)^{k-n+1}
		\psi_{n,\,k}\big((\lambda_j)_{1}^{a_{n,\,k}},(\mu_j)_{1}^{b_{n,\,k}}\big)\frac{t^{k}}{k!},\ n\ge1
	\end{equation}
	and  
	\begin{equation}\label{bdpp0}
		p(0,t)=\sum_{k=1}^{\infty}(-1)^{k+1}
		\psi_{0,\,k}\big((\lambda_j)_{1}^{a_{0,k}},(\mu_j)_{1}^{b_{0,k}}\big)\frac{t^{k}}{k!},
	\end{equation}
	where $\psi_{n,\,k}$'s are as defined in (\ref{psidef}).
\end{remark}
\begin{theorem}
	The state probabilities given in (\ref{pn}) and (\ref{p0}) satisfy the regularity condition.
\end{theorem}
\begin{proof}
	For $t\ge0$, we have
	\begin{align*}
		\sum_{n=0}^{\infty}p^\alpha(n,t)&=\sum_{k=1}^{\infty}(-1)^{k+1}
		\psi_{0,\,k}\big((\lambda_j)_{1}^{a_{0,k}},(\mu_j)_{1}^{b_{0,k}}\big)\frac{t^{k\alpha}}{\Gamma(k\alpha+1)}+\sum_{n=1}^{\infty}\prod_{j=1}^{n-1}\lambda_j\frac{t^{(n-1)\alpha}}{\Gamma((n-1)\alpha+1)}\\
		&\ \ +\sum_{n=1}^{\infty}\sum_{k=n}^{\infty}(-1)^{k-n+1}
		\psi_{n,\,k}\big((\lambda_j)_{1}^{a_{n,\,k}},(\mu_j)_{1}^{b_{n,\,k}}\big)\frac{t^{k\alpha}}{\Gamma(k\alpha+1)}\\
		&=\sum_{k=1}^{\infty}\sum_{n=0}^{k}(-1)^{k-n+1}
		\psi_{n,\,k}\big((\lambda_j)_{1}^{a_{n,\,k}},(\mu_j)_{1}^{b_{n,\,k}}\big)\frac{t^{k\alpha}}{\Gamma(k\alpha+1)}+\sum_{n=1}^{\infty}\prod_{j=1}^{n-1}\lambda_j\frac{t^{(n-1)\alpha}}{\Gamma((n-1)\alpha+1)}\\
		&=1+\sum_{k=1}^{\infty}\sum_{n=0}^{k}(-1)^{k-n+1}
		\psi_{n,\,k}\big((\lambda_j)_{1}^{a_{n,\,k}},(\mu_j)_{1}^{b_{n,\,k}}\big)\frac{t^{k\alpha}}{\Gamma(k\alpha+1)}+\sum_{n=1}^{\infty}\prod_{j=1}^{n}\lambda_j\frac{t^{n\alpha}}{\Gamma(n\alpha+1)},
	\end{align*}
	where using (\ref{psidef}), we have
	\begin{align*}
		\sum_{n=0}^{k}(-1)^{k-n+1}&
		\psi_{n,\,k}\big((\lambda_j)_{1}^{a_{n,\,k}},(\mu_j)_{1}^{b_{n,\,k}}\big)\\
		&=(-1)^{k+1}\psi_{0,\,k}\big((\lambda_j)_{1}^{a_{0,k}},(\mu_j)_{1}^{b_{0,k}}\big)+(-1)^k\psi_{1,\,k}\big((\lambda_j)_{1}^{a_{1,k}},(\mu_j)_{1}^{b_{1,k}}\big)\\
		&\ \ +\sum_{n=2}^{k}\big(\Lambda_n\psi_{n,k-1}\big((\lambda_j)_{1}^{a_{n,k-1}},(\mu_j)_{1}^{b_{n,k-1}}\big)+\lambda_{n-1}\psi_{n-1,k-1}\big((\lambda_j)_{1}^{a_{n-1,k-1}},(\mu_j)_{1}^{b_{n-1,k-1}}\big)\\
		&\ \ +\mu_{n+1}\psi_{n+1,k-1}\big((\lambda_j)_{1}^{a_{n+1,k-1}},(\mu_j)_{1}^{b_{n+1,k-1}}\big)\big)\\
		&=(-1)^{k+1}\psi_{0,\,k}\big((\lambda_j)_{1}^{a_{0,k}},(\mu_j)_{1}^{b_{0,k}}\big)+(-1)^k\psi_{1,\,k}\big((\lambda_j)_{1}^{a_{1,k}},(\mu_j)_{1}^{b_{1,k}}\big)\\
		&\ \ +\sum_{n=2}^{k}\Lambda_n\psi_{n,k-1}\big((\lambda_j)_{1}^{a_{n,k-1}},(\mu_j)_{1}^{b_{n,k-1}}\big)+\sum_{n=1}^{k-1}\lambda_{n}\psi_{n,k-1}\big((\lambda_j)_{1}^{a_{n,k-1}},(\mu_j)_{1}^{b_{n,k-1}}\big)\\
		&\ \ +\sum_{n=3}^{k+1}\mu_{n}\psi_{n,k-1}\big((\lambda_j)_{1}^{a_{n,k-1}},(\mu_j)_{1}^{b_{n,k-1}}\big)\big)\\
		&=(-1)^{k+1}\psi_{0,\,k}\big((\lambda_j)_{1}^{a_{0,k}},(\mu_j)_{1}^{b_{0,k}}\big)+(-1)^k\psi_{1,\,k}\big((\lambda_j)_{1}^{a_{1,k}},(\mu_j)_{1}^{b_{1,k}}\big)\\
		&\ \ +(-1)^{k-1}\mu_2\psi_{2,k-1}\big((\lambda_j)_{1}^{a_{2,k-1}},(\mu_j)_{1}^{b_{2,k-1}}\big)-\lambda_k\psi_{k,k-1}\big((\lambda_j)_{1}^{a_{k,k-1}},(\mu_j)_{1}^{b_{k,k-1}}\big)\\
		&\ \ -(-1)^k\lambda_1\psi_{1,k-1}\big((\lambda_j)_{1}^{a_{1,k-1}},(\mu_j)_{1}^{b_{1,k-1}}\big)\\
		&=-\lambda_k\psi_{k,k-1}\big((\lambda_j)_{1}^{a_{k,k-1}},(\mu_j)_{1}^{b_{k,k-1}}\big)=-\lambda_1\lambda_2\cdots\lambda_k.
	\end{align*}
	Thus, $\sum_{n=0}^{\infty}p^\alpha(n,t)=1$. This completes the proof.
\end{proof}
\begin{remark}
	We note that given a BDP $\{N(t)\}_{t\ge0}$ and an independent inverse $\alpha$-stable subordinator $\{E^\alpha(t)\}_{t\ge0}$, $\alpha\in(0,1)$, following the approach of Meerschaert \textit{et al.} (2011), we can define a time-changed BDP $\{N(E^\alpha(t))\}_{t\ge0}$. In view of Remark \ref{re2.2}, it is equal in distribution with FBDP $\{N^\alpha(t)\}_{t\ge0}$ whenever $0<\alpha<1$. Therefore, the distribution of FBDP can also be derived as follows:
	\begin{equation}\label{lapm}
		p^\alpha(n,t)=\mathrm{Pr}\{N(E^\alpha(t))=n\}=\int_{0}^{\infty}p(n,x)\mathrm{Pr}\{E^\alpha(t)\in\mathrm{d}x\},\ n\ge0,\ \alpha\in(0,1),
	\end{equation}
	where $p(n,x)$ denotes the distribution of BDP. The Laplace transform of (\ref{lapm}) is given by
	\begin{equation*}
		\int_{0}^{\infty}e^{-wt}p^\alpha(n,t)\,\mathrm{d}t=w^{\alpha-1}\int_{0}^{\infty}e^{-w^\alpha x}p(n,x)\,\mathrm{d}x,\ n\ge0,\ w>0,
	\end{equation*}
	where we have used (\ref{ainvsublap}).
	 Hence, using (\ref{bdppn}) and (\ref{bdpp0}), we get
	\begin{equation*}
		\int_{0}^{\infty}e^{-wt}p^\alpha(n,t)\,\mathrm{d}t=\begin{cases}
			\sum_{k=1}^{\infty}(-1)^{k+1}
			\psi_{0,\,k}\big((\lambda_j)_{1}^{a_{0,k}},(\mu_j)_{1}^{b_{0,k}}\big)\frac{1}{w^{k\alpha+1}},\ n=0,\vspace{0.2cm}\\
			\frac{\prod_{j=1}^{n-1}\lambda_j}{w^{(n-1)\alpha+1}}+\sum_{k=n}^{\infty}(-1)^{k-n+1}
			\psi_{n,\,k}\big((\lambda_j)_{1}^{a_{n,\,k}},(\mu_j)_{1}^{b_{n,\,k}}\big)\frac{1}{w^{k\alpha+1}},\ n\ge1,
		\end{cases}
	\end{equation*}
	whose inverse Laplace transform coincides with (\ref{pn}) and (\ref{p0}) for $n\ge1$ and $n=0$, respectively.
\end{remark}

The following result is a consequence of Theorem \ref{thmpmf}. It can be established by adding the series components of the state probabilities of FBDP whenever the rate of birth is equal to the rate of death.
\begin{corollary}
	If the birth and death rates of FBDP are equal at each state, that is, $\lambda_n=\mu_n$ for all $n\ge1$ then its state probabilities are
	\begin{equation}
		p^\alpha(n,t)=\frac{t^{(n-1)\alpha}}{\Gamma((n-1)\alpha+1)}\prod_{j=1}^{n-1}\lambda_j+\sum_{k=n}^{\infty}(-1)^{k-n+1}\xi_{n,\,k}\left(\lambda_1,\lambda_2,\dots,\lambda_{\left[\frac{n+k+1}{2}\right]}\right)\frac{t^{k\alpha}}{\Gamma(k\alpha+1)},\ n\ge1
	\end{equation}
	and its extinction probability is given by
	\begin{equation}
		p^\alpha(0,t)=\sum_{k=1}^{\infty}(-1)^{k+1}\xi_{n,\,k}\left(\lambda_1,\lambda_2,\dots,\lambda_{\left[\frac{n+k+1}{2}\right]}\right)\frac{t^{k\alpha}}{\Gamma(k\alpha+1)},
	\end{equation}
	where $[\cdot]$ is the greatest integer function and $\xi_{n,\,k}\left(\lambda_1,\lambda_2,\dots,\lambda_{\left[\frac{n+k+1}{2}\right]}\right)$ are functions of $\lambda_1,\lambda_2,\dots$, $\lambda_{\left[\frac{n+k+1}{2}\right]}$ such that $\xi_{0,\,0}=0$, $\xi_{1,\,0}=1$, $\xi_{0,\,1}=\lambda_1$ and
	{\small\begin{equation*}
		\xi_{n,\,k+1}\big(\lambda_1,\lambda_2,\dots,\lambda_{\left[\frac{n+k}{2}+1\right]}\big)=\begin{cases}
			0,\ n> k+2,\vspace{0.2cm}\\
			\lambda_1\xi_{1,\,k}\big(\lambda_1,\lambda_2,\dots,\lambda_{\left[\frac{k}{2}+1\right]}\big),\ n=0,\,k\ge1,\vspace{0.2cm}\\
			2\lambda_1\xi_{1,\,k}\big(\lambda_1,\lambda_2,\dots,\lambda_{\left[\frac{k}{2}+1\right]}\big)+\lambda_2\xi_{2,\,k}\big(\lambda_1,\lambda_2,\dots,\lambda_{\left[\frac{k+3}{2}\right]}\big),\ n=1,\,k\ge0,\vspace{0.2cm}\\
			2\lambda_n\xi_{n,\,k}\big(\lambda_1,\lambda_2,\dots,\lambda_{\left[\frac{n+k+1}{2}\right]}\big)+\lambda_{n-1}\xi_{n-1,\,k}\big(\lambda_1,\lambda_2,\dots,\lambda_{\left[\frac{n+k}{2}\right]}\big)\\
			\hspace{2.5cm}+\lambda_{n+1}\xi_{n+1,\,k}\big(\lambda_1,\lambda_2,\dots,\lambda_{\left[\frac{n+k}{2}+1\right]}\big),\ n\ge2,\,k> n-2,\vspace{0.2cm}\\
			\lambda_1\lambda_{2}\cdots\lambda_{n-1},\ n\ge2,\,n-k=2.
		\end{cases}
	\end{equation*}}
\end{corollary} 
\begin{remark}
	Let $\mathscr{T}^\alpha=\inf\{t>0:N^\alpha(t)=0\}$ be the first passage time of FBDP to $0$th state, that is, it is the extinction time of FBDP. Then, its distribution function is given by 
	\begin{equation*}
		\mathrm{Pr}\{\mathscr{T}^\alpha\leq t\}=\mathrm{Pr}\{N^\alpha(t)=0\}=\sum_{k=1}^{\infty}(-1)^{k+1}
		\psi_{0,\,k}\big((\lambda_j)_{1}^{a_{0,k}},(\mu_j)_{1}^{b_{0,k}}\big)\frac{t^{k\alpha}}{\Gamma(k\alpha+1)},\ t>0.
	\end{equation*}
\end{remark}

\begin{proposition}
	The pgf  $G^\alpha(u,t)=\mathbb{E}u^{N^\alpha(t)}$, $|u|\leq1$ of FBDP is
	\begin{equation}\label{fbdppgf}
		G^\alpha(u,t)=u+\sum_{n=1}^{\infty}\sum_{k=n-1}^{\infty}(-1)^{k-n+1}u^{n-1}(u-1)(\lambda_nu-\mu_n)
		\psi_{n,\,k}\big((\lambda_j)_{1}^{a_{n,\,k}},(\mu_j)_{1}^{b_{n,\,k}}\big)\frac{t^{(k+1)\alpha}}{\Gamma((k+1)\alpha+1)},
	\end{equation}
 where $\psi_{n,\,k}$'s are polynomials as defined in (\ref{psidef}).
\end{proposition}
\begin{proof}
	For $t\ge0$, we have $G^\alpha(u,t)=\sum_{n=0}^{\infty}u^np^\alpha(n,t)$. So, from (\ref{initial}), it follows that $G^\alpha(u,0)=u$.
	
	On multiplying  with $u^n$ on both sides of (\ref{diffequ}) and summing over $n=0,1,2\dots$, we get the following fractional differential equation:
	\begin{equation}\label{pgfeq}
		\mathcal{D}_t^\alpha G^\alpha(u,t)=\sum_{n=1}^{\infty}u^{n-1}(u-1)(\lambda_nu-\mu_n)p^\alpha(n,t),\ |u|\leq1,
	\end{equation}
	with initial condition $G^\eta(u,0)=u$.
	
	On taking the Laplace transform on both sides of (\ref{pgfeq}) and using (\ref{frderlap}), we get
	\begin{align*}
		w^\alpha\int_{0}^{\infty}e^{-wt}&G^\alpha(u,t)\,\mathrm{d}t-w^{\alpha-1}u\\
		&=\sum_{n=1}^{\infty}u^{n-1}(u-1)(\lambda_nu-\mu_n)\frac{\prod_{j=1}^{n-1}\lambda_j}{w^{(n-1)\alpha+1}}\\
		&\ \ +\sum_{n=1}^{\infty}u^{n-1}(u-1)(\lambda_nu-\mu_n)\sum_{k=n}^{\infty}
		\psi_{n,\,k}\big((\lambda_j)_{1}^{a_{n,\,k}},(\mu_j)_{1}^{b_{n,\,k}}\big)\frac{(-1)^{k-n+1}}{w^{k\alpha+1}},\ w>0.
	\end{align*}
	So,
	\begin{align*}
		\int_{0}^{\infty}e^{-wt}G^\alpha(u,t)\,\mathrm{d}t&=\frac{u}{w}+\sum_{n=1}^{\infty}u^{n-1}(u-1)(\lambda_nu-\mu_n)\frac{\prod_{j=1}^{n-1}\lambda_j}{w^{n\alpha+1}}\\
		&\ \ +\sum_{n=1}^{\infty}u^{n-1}(u-1)(\lambda_nu-\mu_n)\sum_{k=n}^{\infty}
		\psi_{n,\,k}\big((\lambda_j)_{1}^{a_{n,\,k}},(\mu_j)_{1}^{b_{n,\,k}}\big)\frac{(-1)^{k-n+1}}{w^{(k+1)\alpha+1}},
	\end{align*}
	whose inversion yields the required result.
\end{proof}
\begin{remark}
	On taking the derivative of (\ref{fbdppgf}) with respect to $u$ and substituting $u=1$, we get the mean $\mathbb{E}N^\alpha(t)=({\partial}/{\partial u})G^\alpha(u,t)|_{u=1}$ of FBDP as follows:
	\begin{equation*}
		\mathbb{E}N^\alpha(t)=1+\sum_{n=1}^{\infty}\sum_{k=n-1}^{\infty}(-1)^{k-n+1}(\lambda_n-\mu_n)
		\psi_{n,\,k}\big((\lambda_j)_{1}^{a_{n,\,k}},(\mu_j)_{1}^{b_{n,\,k}}\big)\frac{t^{(k+1)\alpha}}{\Gamma((k+1)\alpha+1)},\ t\ge0.
	\end{equation*}
	In particular, when the birth and death rates are equal then the mean values of FBDP is one.
	
	Further, on taking $r$th order derivative of (\ref{fbdppgf}) with respect to $u$ and substituting $u=1$ gives the $r$th factorial moments of FBDP
	\begin{equation*}
		\mathbb{E}\prod_{j=0}^{r-1}(N^\alpha(t)-j)=\sum_{n=r}^{\infty}\sum_{k=n-1}^{\infty}(-1)^{k-n+1}(\lambda_n-\mu_n)
		\psi_{n,\,k}\big((\lambda_j)_{1}^{a_{n,\,k}},(\mu_j)_{1}^{b_{n,\,k}}\big)\frac{t^{(k+1)\alpha}}{\Gamma((k+1)\alpha+1)},\ r\ge2.
	\end{equation*} 
\end{remark}
\section{FBDP with linear rates}\label{sec4}
 In this section, we consider the case of linear birth and death rates for which the function $\psi_{n,\,k}\big((\lambda_j)_{1}^{a_{n,\,k}},(\mu_j)_{1}^{b_{n,\,k}}\big)$ can be explicitly obtained. 
 Let $\lambda_n=n\lambda$ and $\mu_n=n\mu$ for all $n\ge1$, where $\lambda$ and $\mu$ are some fixed positive constants. Then, we call it the linear FBDP and denote it by $\{\tilde{{N}}^\alpha(t)\}_{t\ge0}$, $\alpha\in(0,1]$. Orsingher and Polito (2011) used the state probability of linear BDP and a similar subordinate relationship as given in (\ref{subr}) to derive the state probabilities of linear FBDP. Here, we obtained a series representation $\tilde{p}^\alpha(n,t)=\sum_{k=0}^{\infty}\tilde{p}_k^\alpha(n,t)$ of its state probabilities $\tilde{p}^\alpha(n,t)=\mathrm{Pr}\{\tilde{{N}}^\alpha(t)=n\}$, $n\ge0$ in three different cases \textit{viz} $\lambda>\mu$, $\lambda<\mu$ and $\lambda=\mu$.

On substituting $\lambda_n=n\lambda$ and $\mu_n=n\mu$ for all $n\ge1$ in (\ref{ncons1}) and solving it recursively, we get the series components $\tilde{p}^\alpha_k(n,t)$, $k\ge0$ as follows:
	\begin{align*}
		\tilde{p}^\alpha_1(0,t)&=\frac{\mu t^\alpha}{\Gamma(\alpha+1)},\ \ 
		\tilde{p}_1^\alpha(1,t)=\frac{-(\lambda+\mu)t^\alpha}{\Gamma(\alpha+1)},\ \ 
		\tilde{p}^\alpha_2(0,t)=\frac{-\mu(\lambda+\mu)t^{2\alpha}}{\Gamma(2\alpha+1)},\\
		\tilde{p}^\alpha_1(2,t)&=\frac{\lambda t^\alpha}{\Gamma(\alpha+1)},\ \ 
		\tilde{p}_2^\alpha(1,t)=\frac{((\lambda+\mu)^2+2\lambda\mu)^{2\alpha}}{\Gamma(2\alpha+1)},\ \ 
		p_3^\alpha(0,t)=\frac{\mu((\lambda+\mu)^2+2\lambda\mu)t^{3\alpha}}{\Gamma(3\alpha+1)},\\
		p_2(2,t)&=\frac{-3\lambda(\lambda+\mu)t^{2\alpha}}{\Gamma(2\alpha+1)},\ \ 
		p_3(1,t)=(-(\lambda+\mu)^3-8\lambda\mu(\lambda+\mu))\frac{t^{3\alpha}}{\Gamma(3\alpha+1)},\\
		p_4(0,t)&=(-\mu(\lambda+\mu)^3-8\lambda\mu^2(\lambda+\mu))\frac{t^{4\alpha}}{\Gamma(4\alpha+1)},\ \ 
		p_2(3,t)= \frac{2\lambda^2t^{2\alpha}}{\Gamma(2\alpha+1)},\dots.
	\end{align*}
	Note that in the case of linear FBDP the sequence of functions $\{\psi_{n,\,k}\}_{n\ge0,\,k\ge0}$ defined in (\ref{psidef}) reduce to the homogeneous polynomials of degree $k$ in $\lambda$ and $\mu$.
	
  The following results are consequences of Theorem \ref{cmpts} that provide the closed form expressions for series components in the case of linear birth and death rates.
\begin{corollary}\label{thm2}
	If the rate of birth is higher than the rate of death in linear FBDP, that is, $\lambda>\mu$ then the series components of its state probabilities are given by
	\begin{equation}\label{thml1}
		\tilde{p}^\alpha_k(n,t)=\begin{cases}
			0,\ n=k=0\ \text{or} \ n-k\ge2,\vspace*{0.2cm}\\
			(-1)^k(\lambda-\mu)^k\sum_{l=0}^{\infty}\big(\frac{\mu}{\lambda}\big)^{l+1}(l^k-(l+1)^k)\frac{t^{k\alpha}}{\Gamma(\alpha k+1)},\ n=0,\ k\ge1,\vspace*{0.2cm}\\
			(-1)^k\frac{(\lambda-\mu)^{k+2}}{\lambda^2}\sum_{l=0}^{\infty}\big(\frac{\mu}{\lambda}\big)^l\binom{n+l}{l}\sum_{r=0}^{n-1}(-1)^r\binom{n-1}{r}\\
			\hspace{5cm}\cdot(r+l+1)^k\frac{t^{k\alpha}}{\Gamma(\alpha k+1)},\ n\ge1,\,n-k\leq1.
		\end{cases}
	\end{equation}
\end{corollary}
\begin{proof}
	If $\lambda>\mu$ then functions $\psi_{n,\,k}$'s as defined in (\ref{psidef}) take the following form:
	\begin{equation}\label{lnemu}
		\psi_{n,\,k}(\lambda,\mu)=\begin{cases}
			0,\ n=k=0\ \text{or} \ n-k\ge2,\vspace{0.2cm}\\
			(\lambda-\mu)^k\sum_{l=0}^{\infty}\big(\frac{\mu}{\lambda}\big)^{l+1}((l+1)^k-l^k),\ n=0,\ k\ge1,\vspace{0,2cm}\\
			(-1)^{n-1}\frac{(\lambda-\mu)^{k+2}}{\lambda^2}\sum_{l=0}^{\infty}\big(\frac{\mu}{\lambda}\big)^l\binom{n+l}{l}\sum_{r=0}^{n-1}(-1)^{r}\binom{n-1}{r}(r+l+1)^k,\ n\ge1,\, n-k\leq1.
		\end{cases}
	\end{equation}
	Now, the proof follows along the similar lines to that of  Theorem \ref{cmpts}. Hence, it is omitted.
\end{proof}
\begin{remark}\label{rem3.1}
	For $\lambda>\mu$, summing (\ref{thml1}) over the range $k$, we get the extinction probability of linear FBDP as follows:
	\begin{align*}
		\tilde{p}^{\alpha}(0,t)&=\sum_{k=1}^{\infty}(-1)^{k}(\lambda-\mu)^{k}\sum_{l=0}^{\infty}\bigg(\frac{\mu}{\lambda}\bigg)^{l+1} (l^k-(l+1)^k)\frac{t^{\alpha k}}{\Gamma(k\alpha+1)}\\
		&=\sum_{l=0}^{\infty}\bigg(\frac{\mu}{\lambda}\bigg)^{l+1}\sum_{k=0}^{\infty}\Bigg(\frac{(-t^\alpha l (\lambda-\mu))^k}{\Gamma(k\alpha+1)}-\frac{(-t^\alpha (l+1)(\lambda-\mu))^k}{\Gamma(k\alpha+1)}\Bigg)\\
		&=\sum_{l=0}^{\infty}\bigg(\frac{\mu}{\lambda}\bigg)^{l+1}(E_{\alpha,1}(-t^\alpha l (\lambda-\mu))-E_{\alpha,1}(-t^\alpha (l+1)(\lambda-\mu)))\\
		&=\sum_{l=0}^{\infty}\bigg(\frac{\mu}{\lambda}\bigg)^{l+1}E_{\alpha,1}(-t^\alpha l (\lambda-\mu))-\sum_{l=1}^{\infty}\bigg(\frac{\mu}{\lambda}\bigg)^{l}E_{\alpha,1}(-t^\alpha l (\lambda-\mu))\\
		&=\frac{\mu}{\lambda}-\frac{(\lambda-\mu)}{\lambda}\sum_{l=1}^{\infty}\bigg(\frac{\mu}{\lambda}\bigg)^{l}E_{\alpha,1}(-t^\alpha l (\lambda-\mu)),
	\end{align*}
	and for $n\ge1$, we get
	\begin{align*}
		\tilde{p}^{\alpha}(n,t)&=\sum_{k=n-1}^{\infty}(-1)^{k}\frac{(\lambda-\mu)^{k+2}}{\lambda^2}\sum_{l=0}^{\infty}\bigg(\frac{\mu}{\lambda}\bigg)^l\binom{n+l}{l}\sum_{r=0}^{n-1}(-1)^r\binom{n-1}{r}(r+l+1)^k\frac{t^{\alpha k}}{\Gamma(k\alpha+1)}\\
		&=\frac{(\lambda-\mu)^{2}}{\lambda^2}\sum_{l=0}^{\infty}\bigg(\frac{\mu}{\lambda}\bigg)^l\binom{n+l}{l}\sum_{r=0}^{n-1}(-1)^r\binom{n-1}{r}\sum_{k=0}^{\infty}\frac{(-t^\alpha(r+l+1)(\lambda-\mu))^k}{\Gamma(k\alpha+1)}\\
		&=\frac{(\lambda-\mu)^{2}}{\lambda^2}\sum_{l=0}^{\infty}\bigg(\frac{\mu}{\lambda}\bigg)^l\binom{n+l}{l}\sum_{r=0}^{n-1}(-1)^r\binom{n-1}{r}E_{\alpha,1}(-t^\alpha (r+l+1)(\lambda-\mu)),
	\end{align*}
	where $E_{\alpha,\,1}(\cdot)$ is the one parameter Mittag-Leffler function defined as follows (see Kilbas \textit{et al.} (2006)):
	\begin{equation}\label{mittag}
		E_{\alpha,\,1}(x)=\sum_{k=0}^{\infty}\frac{x^k}{\Gamma{(k\alpha+1)}},\ x\in\mathbb{R},\,\alpha>0.
	\end{equation}
	 These state probabilities coincide with Eq. (2.13)  and Eq. (3.1) of Orsingher and Polito (2011), respectively
\end{remark}
\begin{corollary}
	If the rate of birth is lower than the rate of death in linear FBDP, that is, $\lambda<\mu$ then the series components of its state probabilities are given by
	\begin{equation*}
		\tilde{p}^\alpha_k(n,t)=\begin{cases}
			0,\,n=k=0\ \mathrm{or} \ n-k\ge2,\vspace{0.2cm}\\
			(-1)^k(\mu-\lambda)^k\sum_{l=0}^{\infty}\big(\frac{\lambda}{\mu}\big)^{l}(l^k-(l+1)^k)\frac{t^{k\alpha}}{\Gamma(k\alpha+1)},\ n=0,\ k\ge1,\vspace{0.2cm}\\
			(-1)^k\big(\frac{\lambda}{\mu}\big)^{n-1}\frac{(\mu-\lambda)^{k+2}}{\mu^2}\sum_{l=0}^{\infty}\big(\frac{\lambda}{\mu}\big)^l\binom{n+l}{l}\sum_{r=0}^{n-1}(-1)^r\binom{n-1}{r}\\
			\hspace{5cm}\cdot(r+l+1)^k\frac{t^{k\alpha}}{\Gamma(k\alpha+1)},\ n\ge1,\,n-k\leq1.
		\end{cases}
	\end{equation*}
\end{corollary}
\begin{proof}
	For $\lambda<\mu$, the functions $\psi_{n,\,k}$'s can be rewritten as follows:
	\begin{equation*}
		\psi_{n,\,k}(\lambda,\mu)=\begin{cases}
			0,\ n=k=0\ \text{or}\ n-k\ge2,\vspace{0.2cm}\\
			(\mu-\lambda)^k\sum_{l=0}^{\infty}\big(\frac{\lambda}{\mu}\big)^{l}((l+1)^k-l^k),\ n=0,\ k\ge1,\vspace{0.2cm}\\
			(-1)^{n-1}\big(\frac{\lambda}{\mu}\big)^{n-1}\frac{(\mu-\lambda)^{k+2}}{\mu^2}\sum_{l=0}^{\infty}\big(\frac{\lambda}{\mu}\big)^l\binom{n+l}{l}\sum_{r=0}^{n-1}(-1)^r\binom{n-1}{r}\\
			\hspace{6cm}\cdot(r+l+1)^k,\ n\ge1,\ n-k\leq1.
		\end{cases}
	\end{equation*}
	Hence, the proof follows along the similar lines to the case of $\lambda>\mu$.
\end{proof}
\begin{remark}
	For $\lambda<\mu$, the extinction probability of linear FBDP is given by 
	\begin{align}
		\tilde{p}^\alpha(0,t)&=\sum_{k=0}^{\infty}(-1)^{k}(\mu-\lambda)^{k}\sum_{l=0}^{\infty}\bigg(\frac{\lambda}{\mu}\bigg)^l (l^k-(l+1)^k)\frac{t^{\alpha k}}{\Gamma(k\alpha+1)}\nonumber\\
		&=\sum_{l=0}^{\infty}\bigg(\frac{\lambda}{\mu}\bigg)^l\sum_{k=0}^{\infty}\Bigg(\frac{(-t^\alpha l (\mu-\lambda))^k}{\Gamma(k\alpha+1)}-\frac{(-t^\alpha (l+1)(\mu-\lambda))^k}{\Gamma(k\alpha+1)}\Bigg)\nonumber\\
		&=\sum_{l=0}^{\infty}\bigg(\frac{\lambda}{\mu}\bigg)^l(E_{\gamma,1}(-t^\alpha l (\mu-\lambda))-E_{\alpha,1}(-t^\alpha (l+1) (\mu-\lambda)))\\
		&=\sum_{l=0}^{\infty}\bigg(\frac{\lambda}{\mu}\bigg)^lE_{\alpha,1}(-t^\alpha l (\mu-\lambda))-\frac{\mu}{\lambda}\sum_{l=1}^{\infty}\bigg(\frac{\lambda}{\mu}\bigg)^lE_{\alpha,1}(-t^\alpha l (\mu-\lambda))\nonumber\\
		&=1-\frac{(\mu-\lambda)}{\lambda}\sum_{l=1}^{\infty}\bigg(\frac{\lambda}{\mu}\bigg)^lE_{\alpha,1}(-t^\alpha l (\mu-\lambda)).\label{lsp0}
	\end{align}
	Moreover, for $n\ge1$, we have
	\begin{align*}
		\tilde{p}^{\alpha}(n,t)&=\sum_{k=0}^{\infty}(-1)^k\bigg(\frac{\lambda}{\mu}\bigg)^{n-1}\frac{(\mu-\lambda)^{k+2}}{\mu^2}\sum_{l=0}^{\infty}\bigg(\frac{\lambda}{\mu}\bigg)^l\binom{n+l}{l}\sum_{r=0}^{n-1}(-1)^r\binom{n-1}{r}(r+l+1)^k\frac{t^{\alpha k}}{\Gamma(k\alpha+1)}\\
		&=\bigg(\frac{\lambda}{\mu}\bigg)^{n-1}\frac{(\mu-\lambda)^{2}}{\mu^2}\sum_{l=0}^{\infty}\bigg(\frac{\lambda}{\mu}\bigg)^l\binom{n+l}{l}\sum_{r=0}^{n-1}(-1)^r\binom{n-1}{r}\sum_{k=0}^{\infty}\frac{(-t^{\alpha}(r+l+1)(\mu-\lambda))^k}{\Gamma(k\alpha+1)}\\
		&=\bigg(\frac{\lambda}{\mu}\bigg)^{n-1}\frac{(\mu-\lambda)^{2}}{\mu^2}\sum_{l=0}^{\infty}\bigg(\frac{\lambda}{\mu}\bigg)^l\binom{n+l}{l}\sum_{r=0}^{n-1}(-1)^r\binom{n-1}{r}E_{\alpha,1}(-t^{\alpha}(r+l+1)(\mu-\lambda)),
	\end{align*}
	which again coincide with Eq. (3.16) of Orsingher and Polito (2011).
\end{remark}
\begin{corollary}
	If the rate of birth and death are same, that is, $\lambda=\mu$ then the series components of the state probabilities of linear FBDP are given by
	\begin{equation*}
		\tilde{p}_k^{\alpha}(n,t)=\begin{cases}
			0,\,n=k=0\ \mathrm{or} \ n-k\ge2,\vspace{0.1cm}\\
			-\frac{k!}{\Gamma(k\alpha+1)}(-\lambda t^{\alpha})^k,\ n=0,\ k\ge1,\vspace{0.1cm}\\
			(-1)^{n+1}\binom{k+1}{n}\frac{ k!}{\Gamma(k\alpha+1)}(-\lambda t^{\alpha})^k,\ n\ge1,\ n-k<2.
		\end{cases}
	\end{equation*}
\end{corollary}
\begin{proof}
	If $\lambda=\mu>0$ then the polynomials defined in (\ref{psidef}) reduce to
	\begin{equation*}
		\psi_{n,\,k}(\lambda)=\begin{cases}
			0,\ n=k=0\ \text{or} \ n-k\ge2,\vspace{0.2cm}\\
			{k!\lambda^k},\ n=0,\ k\ge1,\vspace{0.2cm}\\
			\binom{k+1}{n}{ k!\lambda^k},\ n\ge1,\ n-k<2.
		\end{cases}
	\end{equation*}
	Hence, the proof follows.
\end{proof}
\begin{remark}\label{rem3.3}
	For $\lambda=\mu$, the extinction probability of linear FBDP is given by
	\begin{align*}
		\tilde{p}^{\alpha}(0,t)&=\sum_{k=1}^{\infty}\frac{-k!}{\Gamma(k\alpha+1)}(-\lambda t^{\alpha})^k\\
		&=1-\sum_{k=0}^{\infty}\frac{(-\lambda t^\alpha)^k }{\Gamma(k\alpha+1)}\int_{0}^{\infty}e^{-x}x^k\,\mathrm{d}x\\
		&=1-\int_{0}^{\infty}e^{-x}\sum_{k=0}^{\infty}\frac{(-\lambda t^\alpha)^k x^k}{\Gamma(k\alpha+1)}\,\mathrm{d}x\\
		&=1-\int_{0}^{\infty}e^{-x}E_{\alpha,1}(-\lambda t^{\alpha} x)\,\mathrm{d}x,
	\end{align*}
	where $E_{\alpha,1}(\cdot)$ is the one parameter Mittag-Leffler function defined in (\ref{mittag}).	
	 Also, for $n\ge1$, we have
	\begin{align*}
		\tilde{p}^{\alpha}(n,t)&=\sum_{k=0}^{\infty}(-1)^{n+1}\binom{k+1}{n}\frac{ k!}{\Gamma(k\alpha+1)}(-\lambda t^{\alpha})^k\\
		&=\frac{(-\lambda)^{n-1}}{n!}\sum_{k=0}^{\infty}\frac{k! (-t^\alpha)^k}{\Gamma(k\alpha+1)}\frac{(k+1)!}{(k+1-n)!}{\lambda}^{k-n+1}\\
		&=\frac{(-\lambda)^{n-1}}{n!}\sum_{k=0}^{\infty}\frac{k! (-t^\alpha)^k}{\Gamma(k\alpha+1)}\frac{\mathrm{d}^n}{\mathrm{d}\lambda^n}\lambda^{k+1}\\
		&=\frac{(-\lambda)^{n-1}}{n!}\frac{\mathrm{d}^n}{\mathrm{d}\lambda^n}\bigg(\lambda\sum_{k=0}^{\infty}\frac{k! }{\Gamma(k\alpha+1)}(-\lambda t^\alpha )^k \bigg)\\
		&=\frac{(-\lambda)^{n-1}}{n!}\frac{\mathrm{d}^n}{\mathrm{d}\lambda^n}(\lambda(1-\tilde{p}^{\alpha}(0,t))),
	\end{align*}
	which agree with Eq. (2.26)  and Eq. (3.20) of Orsingher and Polito (2011), respectively.
\end{remark}

\begin{remark}
	Let $\tilde{\mathscr{T}}^\alpha$ denote the time of extinction for linear FBDP, that is, $\tilde{\mathscr{T}}^\alpha=\inf\{t>0:\tilde{{N}}^\alpha(t)=0\}$, the first passage time  of linear FBDP to state $n=0$. In linear FBDP, there is no possibility of immigration. So, its distribution function is given by $\mathrm{Pr}\{\tilde{\mathscr{T}}^\alpha\leq t\}=\mathrm{Pr}\{\tilde{{N}}^\alpha(t)=0\}$, $t\ge0$. In view of Remark \ref{rem3.1}-Remark \ref{rem3.3}, we get
	\begin{equation*}
		\mathrm{Pr}\{\tilde{\mathscr{T}}^\alpha<\infty\}=\lim_{t\to\infty}\mathrm{Pr}\{\tilde{\mathscr{T}}^\alpha\leq t\}=\begin{cases}
			\frac{\mu}{\lambda},\ \lambda>\mu,\\
			1,\ \lambda\leq\mu.
		\end{cases}
	\end{equation*}
	Thus, in linear FBDP, population gets extinct in finite time with probability one whenever $\lambda\leq\mu$. For $\lambda>\mu$, there is no certainty for the extinction of population in finite time.
\end{remark}
\begin{remark}
	On substituting $\lambda_n=n\lambda$ and $\mu_n=n\mu$ in (\ref{pgfeq}), we get the governing equation for the pgf $\tilde{G}^\alpha(u,t)=\mathbb{E}u^{\tilde{{N}}^\alpha(t)}$, $|u|\leq1$ of linear FBDP as follows:
	\begin{equation*}
		\frac{\partial^\alpha}{\partial t^\alpha}\tilde{G}^\alpha(u,t)=(u-1)(\lambda u-\mu)\frac{\partial}{\partial u}\tilde{G}^\alpha(u,t),\ \ \tilde{G}^\alpha(u,0)=u.
	\end{equation*}
	Its mean and variance are given by 
	\begin{equation}\label{meanfbdp}
		\mathbb{E}\tilde{{N}}^\alpha(t)=E_{\alpha,1}((\lambda-\mu)t^\alpha),\ t\ge0
	\end{equation} and
	\begin{equation*}
		\mathbb{V}\mathrm{ar}\tilde{N}^\alpha(t)=\frac{2\lambda}{\lambda-\mu}E_{\alpha,1}(2(\lambda-\mu)t^\alpha)-\frac{\lambda+\mu}{\lambda-\mu}E_{\alpha,1}((\lambda-\mu)t^\alpha)-(E_{\alpha,1}((\lambda-\mu)t^\alpha))^2,\ t\ge0,
	\end{equation*} 
	which agree with Eq. (4.3) and Eq. (4.11) of Orsingher and Polito (2011), respectively.
\end{remark}
\subsection{Cumulative births in linear FBDP}
Here, we discuss some results related to the cumulative births in linear FBDP.
 Let $\tilde{B}(t)$ denote the number of births in linear BDP $\{\tilde{N}(t)\}_{t\ge0}$ up to time $t\ge0$. Then, the joint distribution $\tilde{p}(n,b,t)=\mathrm{Pr}\{\tilde{{N}}(t)=n,\tilde{B}(t)=b\}$, $n\ge0$, $b\ge1$ of bi-variate process $\{(\tilde{{N}}(t),\tilde{B}(t))\}_{t\ge0}$ solves the following system of differential equations (see Vishwakarma and Kataria (2024), Eq. (21) for the case $k_1=k_2=1$):
\begin{equation*}
	\mathcal{D}_t\tilde{p}(b,n,t)=-n(\lambda+\mu)\tilde{p}(b,n,t)+(n-1)\lambda \tilde{p}(b-1,n-1,t)+(n+1)\mu \tilde{p}(b,n+1,t),\ b\ge1,\,n\ge0,
\end{equation*}
with initial condition $p(1,1,0)=1$. 

Let us consider a bi-variate process $\{(\tilde{{N}}^\alpha(t),\tilde{B}^\alpha(t))\}_{t\ge0}$, $0<\alpha\leq1$ whose joint distribution $\tilde{p}^\alpha(n,b,t)=\mathrm{Pr}\{\tilde{{N}}^\alpha(t)=n,\tilde{B}^\alpha(t)=b\}$, $n\ge0$, $b\ge1$ satisfies the following system of fractional differential equations:
 \begin{equation*}
	\mathcal{D}_t^\alpha\tilde{p}^\alpha(b,n,t)=-n(\lambda+\mu)\tilde{p}^\alpha(b,n,t)+(n-1)\lambda \tilde{p}^\alpha(b-1,n-1,t)+(n+1)\mu \tilde{p}^\alpha(b,n+1,t),\ b\ge1,\,n\ge0,
\end{equation*}
with $\tilde{p}^\alpha(1,1,0)=1$, where $\mathcal{D}_t^\alpha$ is the Caputo fractional derivative as defined in (\ref{caputoder}).  
 
For $\alpha=1$, $\tilde{B}^\alpha(t)$ reduces to $\tilde{B}(t)$. The process $\{\tilde{B}^\alpha(t)\}_{t\ge0}$, $0<\alpha\leq1$ plays an important in several applications of the linear FBDP. For example, in an epidemic model under accelerating conditions, if $\tilde{N}^\alpha(t)$ denotes the total number of cases affected by a rapidly spreading disease at time $t\ge0$ then $\tilde{B}^\alpha(t)$ denotes the total number of cases recorded by time $t$. If  the process $\tilde{N}^\alpha(t)$ is such that the epidemic ends in some finite time then $\tilde{B}^\alpha(t)$ is a measure of epidemic severity.

The joint pgf $\tilde{G}^\alpha(u,v,t)=\mathbb{E}u^{\tilde{{N}}^\alpha(t)}v^{\tilde{B}^\alpha(t)}$, $|u|\leq1$, $|v|\leq1$ solves 
\begin{equation}\label{cumbpgf}
\mathcal{D}_t^\alpha\tilde{G}^\alpha(u,v,t)=(\lambda u(uv-1)+\mu(1-u))\frac{\partial}{\partial u}\tilde{G}^\alpha(u,v,t),
\end{equation}
with $\tilde{G}^\alpha(u,v,0)=uv$.

On taking $u=e^{\theta_u}$ and $v=e^{\theta_v}$ in (\ref{cumbpgf}) for appropriate choices of $\theta_u$ and $\theta_v$, we get the following fractional differential equation that governs the joint cumulant generating function $\tilde{K}^\alpha(\theta_u,\theta_v,t)=\ln\tilde{G}^\alpha\left(e^{\theta_u},e^{\theta_v},t\right)$ of $(\tilde{N}^\alpha(t),\tilde{B}^\alpha(t))$:
\begin{equation}\label{cgf}
\mathcal{D}_t^\alpha\tilde{K}^\alpha(\theta_u,\theta_v,t)=\left(\lambda\big(e^{\theta_u+\theta_v}-1\big)+\mu\big(e^{-\theta_u}-1\big)\right)\mathcal{D}_{\theta_u}\tilde{K}^\alpha(\theta_u,\theta_v,t),
\end{equation}
with initial condition $\tilde{K}^\alpha(\theta_u,\theta_v,0)=\theta_u+\theta_v$.

Now, on substituting the following expansion of cumulant generating function (see Kendall (1948), Eq. (35)):
\begin{equation*}
	\tilde{K}^\alpha(\theta_u,\theta_v,t)=\theta_u\mathbb{E}\tilde{{N}}^\alpha(t)+\theta_v\mathbb{E}\tilde{B}^\alpha(t)+\frac{\theta_u^2}{2}\mathbb{V}\mathrm{ar}\tilde{{N}}^\alpha(t)+\frac{\theta_v^2}{2}\mathbb{V}\mathrm{ar}\tilde{B}^\alpha(t)+\theta_u\theta_v\mathbb{C}\mathrm{ov}(\tilde{N}^\alpha(t),\tilde{B}^\alpha(t))+\dots
\end{equation*}
in (\ref{cgf}), we have
\begin{align}\label{cgfexp}
\mathcal{D}_t^\alpha\bigg(\theta_u\mathbb{E}&\tilde{{N}}^\alpha(t)+\theta_v\mathbb{E}\tilde{B}^\alpha(t)+\frac{\theta_u^2}{2}\mathbb{V}\mathrm{ar}\tilde{{N}}^\alpha(t)+\frac{\theta_v^2}{2}\mathbb{V}\mathrm{ar}\tilde{B}^\alpha(t)+\theta_u\theta_v\mathbb{C}\mathrm{ov}(\tilde{N}^\alpha(t),\tilde{B}^\alpha(t))+\dots \bigg)\nonumber\\
&=\sum_{k=0}^{\infty}\left(\frac{\lambda(\theta_u+\theta_v)^k+\mu(-\theta_u)^k}{k!}\right)\big(\mathbb{E}\tilde{{N}}^\alpha(t)+{\theta_u}\mathbb{V}\mathrm{ar}\tilde{{N}}^\alpha(t)+\theta_v\mathbb{C}\mathrm{ov}(\tilde{N}^\alpha(t),\tilde{B}^\alpha(t))+\dots\big).
\end{align}

On comparing the coefficients on both sides of (\ref{cgfexp}), we get the following differential equations:
{\small\begin{align}
	\mathcal{D}_t^\alpha\mathbb{V}\mathrm{ar}\tilde{{N}}^\alpha(t)&=2(\lambda-\mu)\mathbb{V}\mathrm{ar}\tilde{{N}}^\alpha(t)+(\lambda+\mu)\mathbb{E}\tilde{{N}}^\alpha(t),\ \ \mathbb{V}\mathrm{ar}\tilde{{N}}^\alpha(0)=0,\label{0}\\
	\mathcal{D}_t^\alpha\mathbb{E}\tilde{B}^\alpha(t)&=\lambda\mathbb{E}\tilde{N}^\alpha(t),\ \ \mathbb{E}\tilde{B}^\alpha(0)=1,\label{1}\\
	\mathcal{D}_t^\alpha\mathbb{C}\mathrm{ov}(\tilde{N}^\alpha(t),\tilde{B}^\alpha(t))&=\lambda(\mathbb{E}\tilde{N}^\alpha(t)+\mathbb{V}\mathrm{ar}\tilde{N}^\alpha(t))+(\lambda-\mu)\mathbb{C}\mathrm{ov}(\tilde{N}^\alpha(t),\tilde{B}^\alpha(t)),\ \ \mathbb{C}\mathrm{ov}(\tilde{N}^\alpha(t),\tilde{B}^\alpha(0))=0\label{2}
\end{align}}
and
\begin{equation}\label{3}
	\mathcal{D}_t^\alpha\mathbb{V}\mathrm{ar}\tilde{B}^\alpha(t)=\lambda\mathbb{E}\tilde{N}^\alpha(t)+2\lambda\mathbb{C}\mathrm{ov}(\tilde{N}^\alpha(t),\tilde{B}^\alpha(t)),\ \ \mathbb{V}\mathrm{ar}\tilde{B}^\alpha(0)=0.
\end{equation}
 By substituting (\ref{meanfbdp}) in (\ref{0}), we get
	\begin{equation}\label{pf0}
		\mathcal{D}_t^\alpha\mathbb{V}\mathrm{ar}\tilde{{N}}^\alpha(t)=2(\lambda-\mu)\mathbb{V}\mathrm{ar}\tilde{{N}}^\alpha(t)+(\lambda+\mu)E_{\alpha,1}((\lambda-\mu)t^\alpha),\ \ \mathbb{V}\mathrm{ar}\tilde{{N}}^\alpha(0)=0.
	\end{equation}
	On taking the Laplace transform on both sides of (\ref{pf0}) and using (\ref{frderlap}), we get
	\begin{equation*}
		w^\alpha\int_{0}^{\infty}e^{-wt}\mathbb{V}\mathrm{ar}\tilde{{N}}^\alpha(t)\,\mathrm{d}t=2(\lambda-\mu)\int_{0}^{\infty}e^{-wt}\mathbb{V}\mathrm{ar}\tilde{{N}}^\alpha(t)\,\mathrm{d}t+\frac{(\lambda+\mu)w^{\alpha-1}}{w^\alpha-(\lambda-\mu)},
	\end{equation*}
	where we have used the Laplace transform of Mittag-Leffler function given as follows (see Kilbas \textit{et al.} (2006)):
	\begin{equation}\label{mittaglap}
		\int_{0}^{\infty}e^{-wt}E_{\alpha,1}(ct^\alpha)\,\mathrm{d}t=\frac{w^{\alpha-1}}{w^\alpha-c},\ c\in\mathbb{R},\ w>0.
	\end{equation}
	So,
	\begin{equation}\label{pvarlap}
		\int_{0}^{\infty}e^{-wt}\mathbb{V}\mathrm{ar}\tilde{{N}}^\alpha(t)\,\mathrm{d}t=\frac{\lambda+\mu}{\lambda-\mu}\left(\frac{w^{\alpha-1}}{w^\alpha-2(\lambda-\mu)}-\frac{w^{\alpha-1}}{w^\alpha-(\lambda-\mu)}\right),\ w>0,
	\end{equation}
	which on taking the inverse Laplace transform yields 
	\begin{equation}\label{varn}
		\mathbb{V}\mathrm{ar}\tilde{{N}}^\alpha(t)=\frac{\lambda+\mu}{\lambda-\mu}(E_{\alpha,1}(2(\lambda-\mu)t^\alpha)-E_{\alpha,1}((\lambda-\mu)t^\alpha)),\ t\ge0.
	\end{equation}
	It is an alternate representation for the variance of linear FBDP.
	
	By substituting (\ref{meanfbdp}) in (\ref{1}), we get
	\begin{equation}\label{p1}
		\mathcal{D}_t^\alpha\mathbb{E}\tilde{B}^\alpha(t)=\lambda E_{\alpha,1}((\lambda-\mu)t^\alpha),\ \ \mathbb{E}\tilde{B}^\alpha(0)=1.
	\end{equation}
	On taking the Laplace transform on both sides of (\ref{p1}) and using (\ref{frderlap}), we get
	\begin{equation*}
		w^\alpha\int_{0}^{\infty}e^{-wt}\mathbb{E}\tilde{B}^\alpha(t)\,\mathrm{d}t-w^{\alpha-1}=\frac{\lambda w^{\alpha-1}}{w^\alpha-(\lambda-\mu)},\ w>0.
	\end{equation*}	
	So,
	\begin{equation*}
		\int_{0}^{\infty}e^{-wt}\mathbb{E}\tilde{B}^\alpha(t)\,\mathrm{d}t=\frac{1}{w}+\frac{\lambda}{\lambda-\mu}\left(\frac{w^{\alpha-1}}{w^\alpha-(\lambda-\mu)}-\frac{1}{w}\right),
	\end{equation*}
	whose inversion yields
	\begin{equation}\label{meanb}
		\mathbb{E}\tilde{B}^\alpha(t)=1+\frac{\lambda}{\lambda-\mu}(E_{\alpha,1}((\lambda-\mu)t^\alpha)-1),\ t\ge0.
	\end{equation}
	
	 On taking the Laplace transform on both sides of (\ref{2}), and using (\ref{meanfbdp}) and (\ref{pvarlap}), we get
	\begin{align}
		\int_{0}^{\infty}e^{-wt}&\mathbb{C}\mathrm{ov}(\tilde{N}^\alpha(t),\tilde{B}^\alpha(t))\,\mathrm{d}t\nonumber\\
		&=\frac{\lambda}{w^\alpha-(\lambda-\mu)}\left(\frac{w^{\alpha-1}}{w^\alpha-(\lambda-\mu)}+\frac{\lambda+\mu}{\lambda-\mu}\left(\frac{w^{\alpha-1}}{w^\alpha-2(\lambda-\mu)}-\frac{w^{\alpha-1}}{w^\alpha-(\lambda-\mu)}\right)\right)\nonumber\\
		&=\frac{\lambda(\lambda+\mu)}{(\lambda-\mu)^2}\left(\frac{w^{\alpha-1}}{w^\alpha-2(\lambda-\mu)}-\frac{w^{\alpha-1}}{w^\alpha-(\lambda-\mu)}\right)-\frac{2\lambda\mu}{\lambda-\mu}\frac{w^{\alpha-1}}{(w^\alpha-(\lambda-\mu))^2},\ w>0.\label{pcovlap}
	\end{align}
	On taking the inverse Laplace transform on both sides of (\ref{pcovlap}), we get
	\begin{align}\label{pcovnb}
		\mathbb{C}\mathrm{ov}(\tilde{N}^\alpha(t),\tilde{B}^\alpha(t))\,\mathrm{d}t&=\frac{\lambda(\lambda+\mu)}{(\lambda-\mu)^2}(E_{\alpha,1}(2(\lambda-\mu)t^\alpha)-E_{\alpha,1}((\lambda-\mu)t^\alpha))\nonumber\\
		&\ \ -\frac{2\lambda\mu}{\lambda-\mu}\int_{0}^{t}x^{\alpha-1}E_{\alpha,\alpha}((\lambda-\mu)x^\alpha)E_{\alpha,1}((t-x)^\alpha(\lambda-\mu))\,\mathrm{d}x,
	\end{align}
	which involves the convolution of one parameter Mittag-Leffler function (for definition see (\ref{mittag})) with $x^{\alpha-1}E_{\alpha,\alpha}(cx)$, $c\in\mathbb{R}$. Here, $E_{\alpha,\alpha}(\cdot)$ is the two parameter Mittag-Leffler function defined by (see Kilbas \textit{et al.} (2006), p. 42)
	\begin{equation*}
		E_{\alpha,\beta}(x)=\sum_{k=0}^{\infty}\frac{x^k}{\Gamma(k\alpha+\beta)},\ x\in\mathbb{R},\ \alpha>0,\,\beta>0.
	\end{equation*}
	For $\beta=1$, it reduces to the one parameter Mittag-Leffler function.
	Also, we have used the following result to obtain (\ref{pcovnb}) (see Kilbas \textit{et al.} (2006), p. 47):
	\begin{equation*}
		\int_{0}^{\infty}e^{-wt}t^{\beta-1}E_{\alpha,\beta}(ct^\alpha)\,\mathrm{d}t=\frac{w^{\alpha-\beta}}{w^\alpha-c},\ |cw^{-\alpha}|<1,\ w>0.
	\end{equation*}
	The integral in (\ref{pcovnb}) can be evaluated as follows:
	\begin{align}
		\int_{0}^{t}x^{\alpha-1}E_{\alpha,\alpha}((\lambda-\mu)x^\alpha)&E_{\alpha,1}((t-x)^\alpha(\lambda-\mu))\,\mathrm{d}x\nonumber\\
		&=\sum_{l=0}^{\infty}\sum_{m=0}^{\infty}\frac{(\lambda-\mu)^{l+m}}{\Gamma(l\alpha+\alpha)\Gamma(m\alpha+1)}\int_{0}^{t}x^{\alpha(l+1)-1}(t-x)^{\alpha m}\,\mathrm{d}x\nonumber\\
		&=\sum_{l=0}^{\infty}\sum_{m=0}^{\infty}\frac{(\lambda-\mu)^{l+m}}{\Gamma(l\alpha+\alpha)\Gamma(m\alpha+1)}t^{\alpha(l+m+1)}\int_{0}^{1}y^{\alpha(l+1)-1}(1-y)^{\alpha m}\,\mathrm{d}y\nonumber\\
		&=\sum_{l=0}^{\infty}\sum_{m=0}^{\infty}\frac{(\lambda-\mu)^{l+m}}{\Gamma(l\alpha+\alpha)\Gamma(m\alpha+1)}t^{\alpha(l+m+1)}\frac{\Gamma(\alpha(l+1))\Gamma(\alpha m+1)}{\Gamma(\alpha(l+m+1)+1)}\nonumber\\
		&=\sum_{l=0}^{\infty}\sum_{k=l}^{\infty}\frac{(\lambda-\mu)^k}{\Gamma(\alpha k+\alpha+1)}t^{\alpha k+\alpha}\nonumber\\
		&=\sum_{k=0}^{\infty}\frac{(\lambda-\mu)^k(k+1)}{\Gamma(\alpha k+\alpha+1)}t^{\alpha k+\alpha}=\frac{t^\alpha}{\alpha}E_{\alpha,\alpha}(t^\alpha(\lambda-\mu)).\label{conint}
	\end{align}
	Thus, 
	\begin{equation}\label{covbn}
		\mathbb{C}\mathrm{ov}(\tilde{N}^\alpha(t),\tilde{B}^\alpha(t))=\frac{\lambda(\lambda+\mu)}{(\lambda-\mu)^2}(E_{\alpha,1}(2(\lambda-\mu)t^\alpha)-E_{\alpha,1}((\lambda-\mu)t^\alpha))-\frac{2\lambda\mu t^\alpha}{(\lambda-\mu)\alpha}E_{\alpha,\alpha}(t^\alpha(\lambda-\mu)),\, t\ge0.
	\end{equation}
	
	 On taking the Laplace transform on both sides (\ref{3}) and using (\ref{pcovlap}), we get
	\begin{align*}
		\int_{0}^{\infty}e^{-wt}\mathbb{V}\mathrm{ar}\tilde{B}^\alpha(t)\,\mathrm{d}t&=\frac{\lambda}{w(w^\alpha-(\lambda-\mu))}+\frac{2\lambda^2(\lambda+\mu)}{(\lambda-\mu)^2}\left(\frac{1}{w(w^\alpha-2(\lambda-\mu))}-\frac{1}{w(w^\alpha-(\lambda-\mu))}\right)\\
		&\ \ -\frac{4\lambda^2\mu}{\lambda-\mu}\frac{1}{w(w^\alpha-(\lambda-\mu))^2}\\
		&=\frac{\lambda}{\lambda-\mu}\left(\frac{w^{\alpha-1}}{w^\alpha-(\lambda-\mu)}-\frac{1}{w}\right)+\frac{\lambda^2(\lambda+\mu)}{(\lambda-\mu)^3}\left(\frac{w^{\alpha-1}}{w^\alpha-2(\lambda-\mu)}-\frac{1}{w}\right)\\
		&\ \ -\frac{2\lambda^2(\lambda+\mu)}{(\lambda-\mu)^3}\left(\frac{w^{\alpha-1}}{w^\alpha-(\lambda-\mu)}-\frac{1}{w}\right)\\
		&\ \ -\frac{4\lambda^2\mu}{(\lambda-\mu)^2}\left(\frac{w^{\alpha-1}}{(w^\alpha-(\lambda-\mu))^2}-\frac{1}{\lambda-\mu}\left(\frac{w^{\alpha-1}}{w^\alpha-(\lambda-\mu)}-\frac{1}{w}\right)\right)\\
		&=\frac{\lambda^2(\lambda+\mu)}{(\lambda-\mu)^3}\left(\frac{w^{\alpha-1}}{w^\alpha-2(\lambda-\mu)}-\frac{1}{w}\right)-\frac{\lambda(\lambda+\mu)}{(\lambda-\mu)^2}\left(\frac{w^{\alpha-1}}{w^\alpha-(\lambda-\mu)}-\frac{1}{w}\right)\\
		&\ \ -\frac{4\lambda^2\mu}{(\lambda-\mu)^2}\frac{w^{\alpha-1}}{(w^\alpha-(\lambda-\mu))^2}.
	\end{align*}
	Its inverse Laplace transform is given by
	\begin{align}
		\mathbb{V}\mathrm{ar}\tilde{B}^\alpha(t)&=\frac{\lambda^2(\lambda+\mu)}{(\lambda-\mu)^3}(E_{\alpha,1}(2(\lambda-\mu)t^\alpha)-1)-\frac{\lambda(\lambda+\mu)}{(\lambda-\mu)^2}(E_{\alpha,1}((\lambda-\mu)t^\alpha)-1)\nonumber\\
		&\ \ -\frac{4\lambda^2\mu}{(\lambda-\mu)^2}\int_{0}^{t}x^{\alpha-1}E_{\alpha,\alpha}((\lambda-\mu)x^\alpha)E_{\alpha,1}((t-x)^\alpha(\lambda-\mu))\,\mathrm{d}x\nonumber\\
		&=\frac{\lambda^2(\lambda+\mu)}{(\lambda-\mu)^3}(E_{\alpha,1}(2(\lambda-\mu)t^\alpha)-1)-\frac{\lambda(\lambda+\mu)}{(\lambda-\mu)^2}(E_{\alpha,1}((\lambda-\mu)t^\alpha)-1)\nonumber\\
		&\ \ -\frac{4\lambda^2\mu t^\alpha}{(\lambda-\mu)^2\alpha}E_{\alpha,\alpha}((\lambda-\mu)t^\alpha),\ t\ge0,\label{varb}
	\end{align}
	where we have used (\ref{conint}) in the last step. This completes the proof.
	
	In the case $\lambda\ne\mu$, for different values of $\alpha$, the variation of correlation coefficient 
	\begin{equation*}
		\mathrm{Corr}(\tilde{N}^\alpha(t),\tilde{B}^\alpha(t))=\frac{\mathbb{C}\mathrm{ov}(\tilde{N}^\alpha(t),\tilde{B}^\alpha(t))}{\sqrt{\mathbb{V}\mathrm{ar}\tilde{N}^\alpha(t)\mathbb{V}\mathrm{ar}\tilde{B}^\alpha(t)}},\ t\ge0
	\end{equation*}
	of $\tilde{N}^\alpha(t)$ and $\tilde{B}^\alpha(t)$ with respect to time is shown in Figure \ref{fig1}. It is observed that they are positively correlated. Further, if $\lambda>\mu$ then the correlation is getting stronger with time and eventually converges to one. Moreover, the convergence to one is faster for the smaller values of $\alpha$. For $\lambda<\mu$, the correlation increases initially for some time then rapidly decreases. Also, the increase and decrease in correlation coefficient is slow for smaller values of $\alpha$.
	\begin{figure}[ht!]
		\includegraphics[width=16cm]{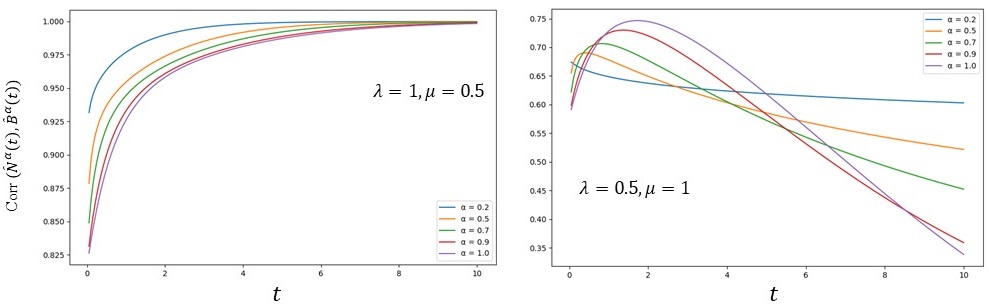}
		\caption{Correlation coefficient of $\tilde{N}^\alpha(t)$ and $\tilde{B}^\alpha(t)$ for different values of $\alpha$}\label{fig1}
	\end{figure}
\begin{remark}\label{nfcumb}
	For $\alpha=1$, we have $E_{1,1}(x)=e^{x}$, $x\in\mathbb{R}$. Thus, for $t\ge0$, the expressions obtained in (\ref{varn}),  (\ref{meanb}), (\ref{covbn}) and (\ref{varb}) reduce to
	\begin{align*}
		\mathbb{V}\mathrm{ar}\tilde{N}(t)&=\frac{\lambda+\mu}{\lambda-\mu}e^{(\lambda-\mu)t}(e^{(\lambda-\mu)t}-1),\\
		\mathbb{E}\tilde{B}(t)&=1+\frac{\lambda}{\lambda-\mu}(e^{(\lambda-\mu)t}-1),\\
		\mathbb{C}\mathrm{ov}(\tilde{N}(t),\tilde{B}(t))&=\frac{\lambda(\lambda+\mu)}{(\lambda-\mu)^2}e^{(\lambda-\mu)t}(e^{(\lambda-\mu)t}-1)-\frac{2\lambda\mu t}{\lambda-\mu}e^{(\lambda-\mu)t}
	\end{align*}
	and
	\begin{equation*}
		\mathbb{V}\mathrm{ar}\tilde{B}(t)=\frac{\lambda^2(\lambda+\mu)}{(\lambda-\mu)^3}(e^{2(\lambda-\mu)t}-1)-\frac{\lambda(\lambda+\mu)}{(\lambda-\mu)^2}(e^{(\lambda-\mu)}-1)-\frac{4\lambda^2\mu t}{(\lambda-\mu)^2}e^{(\lambda-\mu)t},
	\end{equation*}
	respectively, which agree with the results given in Proposition 5 of Vishwakarma and Kataria (2024a) for the case $k_1=k_2=1$.
\end{remark}
\begin{remark}
As done to established the time-changed relationship (\ref{subr}), it can be shown that 
\begin{equation}\label{cumbtcr}
	(\tilde{N}^\alpha(t),\tilde{B}^\alpha(t))\overset{d}{=}(\tilde{N}(T^{2\alpha}(t)),\tilde{B}(T^{2\alpha}(t))),\ t\ge0,
\end{equation}
where the process $\{T^{2\alpha}(t)\}_{t\ge0}$ is as defined in Theorem \ref{thmsub}, and it is independent of $\{(\tilde{N}(t),\tilde{B}(t))\}_{t\ge0}$. Therefore, we have $\tilde{B}^\alpha(t)\overset{d}{=}\tilde{B}(T^{2\alpha}(t))$, $t\ge0$, which can also be used to compute the mean and variance of $\tilde{B}^\alpha(t)$ using Remark \ref{nfcumb}. For example, we have $\mathbb{E}\tilde{B}^\alpha(t)=\int_{0}^{\infty}\mathbb{E}\tilde{B}(x)\mathrm{Pr}\{T^{2\alpha}(t)\in\mathrm{d}x\}$, whose Laplace transform is
\begin{align*}
	\int_{0}^{\infty}e^{-wt}\mathbb{E}\tilde{B}^\alpha(t)\,\mathrm{d}t&=w^{\alpha-1}\int_{0}^{\infty}\left(1+\frac{\lambda}{\lambda-\mu}(e^{(\lambda-\mu)x}-1)\right)e^{-w^\alpha x}\,\mathrm{d}x\\
	&=\frac{1}{w}+\frac{\lambda}{\lambda-\mu}\left(\frac{w^{\alpha-1}}{w^\alpha-(\lambda-\mu)}-\frac{1}{w}\right),\ w>0.
\end{align*}
By using (\ref{mittaglap}), its inverse Laplace transform coincides with (\ref{meanb}).
\end{remark}
\begin{remark}
	Let $\tilde{D}^\alpha(t)$ denote the number of deaths in linear FBDP up to time $t\ge0$. Then, $\tilde{D}^\alpha(t)=\tilde{N}^\alpha(t)-\tilde{B}^\alpha(t)$, and 
	$\mathbb{E}\tilde{D}^\alpha(t)={\mu}{(\lambda-\mu)^{-1}}(E_{\alpha,1}((\lambda-\mu)t^\alpha)-1)$, $t\ge0$. Similarly, other expressions can also be derived. Moreover, it can be observed that $\mathbb{C}\mathrm{ov}(\tilde{B}^\alpha(t),\tilde{D}^\alpha(t))=\mathbb{C}\mathrm{ov}(\tilde{N}^\alpha(t),\tilde{B}^\alpha(t))-\mathbb{V}\mathrm{ar}\tilde{B}^\alpha(t)\ne0$. Thus, the number of births and deaths in linear FBDP are linearly correlated.
\end{remark}
\subsubsection{Asymptotic distribution of cumulative birth process} On using the time-changed relationship (\ref{cumbtcr}), the joint pgf of $(\tilde{N}^\alpha(t),\tilde{B}^\alpha(t))$ is 
$\tilde{G}^\alpha(u,v,t)=\int_{0}^{\infty}\tilde{G}(u,v,x)\mathrm{Pr}\{T^{2\alpha}(t)\in\mathrm{d}x\}$. So, by using (\ref{diffslap}), its Laplace transform is given by
\begin{equation}\label{ascumblap}
	\int_{0}^{\infty}e^{-wt}\tilde{G}^\alpha(u,v,t)\,\mathrm{d}t=w^{\alpha-1}\int_{0}^{\infty}\tilde{G}(u,v,x)e^{-w^\alpha x}\mathrm{d}x,\ w>0,
\end{equation}
where $\tilde{G}(u,v,t)$ is the joint pgf of  $(\tilde{N}(t),\tilde{B}(t))$. It is given by (see Kendall (1948), Eq. (50)) 
\begin{equation}\label{pgfcumbn}
	\tilde{G}(u,v,t)=v\frac{r_1(v)(r_2(v)-u)+r_2(v)(u-r_1(v))e^{-\lambda v(r_2(v)-r_1(v))t}}{(r_2(v)-u)+(u-r_1(v))e^{-\lambda v(r_2(v)-r_1(v))t}},\ 0\leq u\leq1,\,0\leq v\leq1,
\end{equation}
where $r_1(v)$ and $r_2(v)$ are the roots of $\lambda vu^2-(\lambda+\mu)u+\mu=0$ such that $0<r_1(v)<1<r_2(v)$.

On substituting (\ref{pgfcumbn}) in (\ref{ascumblap}), we get
{\scriptsize\begin{align*}
	\int_{0}^{\infty}&e^{-wt}\tilde{G}^\alpha(u,v,t)\,\mathrm{d}t\\
	&=w^{\alpha-1}v\int_{0}^{\infty}\bigg({r_1(v)-r_2(v)\frac{u-r_1(v)}{u-r_2(v)}e^{-\lambda v(r_2(v)-r_1(v))x}}\bigg)\bigg({1-\frac{u-r_1(v)}{u-r_2(v)}e^{-\lambda v(r_2(v)-r_1(v))x}}\bigg)^{-1}e^{-w^\alpha x}\,\mathrm{d}x\\
	&=w^{\alpha-1}v\sum_{k=0}^{\infty}\int_{0}^{\infty}\bigg(r_1(v)\bigg(\frac{u-r_1(v)}{u-r_2(v)}\bigg)^ke^{-\lambda v(r_2(v)-r_1(v))xk}-r_2(v)\bigg(\frac{u-r_1(v)}{u-r_2(v)}\bigg)^{k+1}e^{-\lambda v(r_2(v)-r_1(v))x(k+1)}\bigg){e^{-w^\alpha x}}\,\mathrm{d}x\\
	&=v\sum_{k=0}^{\infty}\bigg(r_1(v)\bigg(\frac{u-r_1(v)}{u-r_2(v)}\bigg)^k\frac{w^{\alpha-1}}{w^\alpha+\lambda v(r_2(v)-r_1(v))k}-r_2(v)\bigg(\frac{u-r_1(v)}{u-r_2(v)}\bigg)^{k+1}\frac{w^{\alpha-1}}{w^\alpha+\lambda v(r_2(v)-r_1(v))(k+1)}\bigg)\\
	&=v\bigg(\frac{r_1(v)}{w}-(r_2(v)-r_1(v))\sum_{k=1}^{\infty}\bigg(\frac{u-r_1(v)}{u-r_2(v)}\bigg)^k\frac{w^{\alpha-1}}{w^\alpha+\lambda v(r_2(v)-r_1(v))k}\bigg),\ w>0.
\end{align*}}
Its inverse Laplace transform is given by
\begin{equation}\label{jointpgfcumb}
	\tilde{G}^\alpha(u,v,t)=v\bigg(r_1(v)-(r_2(v)-r_1(v))\sum_{k=1}^{\infty}\bigg(\frac{u-r_1(v)}{u-r_2(v)}\bigg)^kE_{\alpha,1}(-\lambda v(r_2(v)-r_1(v))kt^\alpha)\bigg).
\end{equation}

Next, we derive the asymptotic distribution of $\tilde{B}^\alpha(t)$ when the birth rate is less than or equal to the death rate. For $\lambda\leq\mu$, on letting $t\to\infty$ in (\ref{jointpgfcumb}), we get
$\tilde{G}^\alpha(1,v)=\lim_{t\to\infty}\tilde{G}^\alpha(1,v,t)$ $=vr_1(v)$, $0\leq v\leq1$. Thus, the process $\{\tilde{B}^\alpha(t)\}_{t\ge0}$ has stable limiting distribution when $\lambda\leq\mu$ and $\tilde{B}^\alpha(t)$ converges in distribution to a random variable $\tilde{B}^\alpha$ as $t\to\infty$. As $r_1(v)$ is the smaller root of $\lambda vu^2-(\lambda+\mu)u+\mu=0$, the pgf of $\tilde{B}^\alpha$ is given by
\begin{equation*}
	\tilde{G}^\alpha(1,v)=(2\lambda)^{-1}((\lambda+\mu)-\sqrt{(\lambda+\mu)^2-4\lambda\mu v}),\ 0\leq v\leq1.
\end{equation*}
So, by using the generalized binomial theorem, we get
\begin{equation*}
	\tilde{G}^\alpha(1,v)=\frac{\lambda+\mu}{2\lambda}\sum_{k=1}^{\infty}\frac{1\cdot3\cdot5\cdots(2k-3)}{2^kk!}\bigg(\frac{4\lambda\mu v}{(\lambda+\mu)^2}\bigg)^k=\frac{\lambda+\mu}{2\lambda}\sum_{k=1}^{\infty}\frac{(2k)!}{2^{2k}(k!)^2(2k-1)}\bigg(\frac{4\lambda\mu v}{(\lambda+\mu)^2}\bigg)^k.
\end{equation*}
Thus, for $\lambda\leq\mu$, the probability mass function of $\tilde{B}^\alpha$ is given by
\begin{equation*}
	\mathrm{Pr}\{\tilde{B}^\alpha=b\}=\frac{(2b)!}{(b!)^2(2b-1)}\frac{(\lambda\mu)^b}{2\lambda(\lambda+\mu)^{2b-1}},\ b\ge1.
\end{equation*}
Its mean and variance are given as follows:
\begin{equation*}
	\mathbb{E}\tilde{B}^\alpha=\frac{\mu}{\mu-\lambda}\ \ \text{and}\ \ \mathbb{V}\mathrm{ar}\tilde{B}^\alpha=\frac{\lambda\mu(\lambda+\mu)}{(\mu-\lambda)^3}.
\end{equation*}

\section{Extinction time of a time-changed linear BDP}
In this section, we consider a time-changed BDP where the time changes according to an inverse subordinator. For linear birth and death rates such that the birth rate is lower than the death rate, we study the properties of extinction time for the time-changed linear BDP.

 Let $\{S^\phi(t)\}_{t\ge0}$ be a subordinator with Laplace exponent $\phi(\cdot)$ and $\{E^\phi(t)\}_{t\ge0}$ be its first hitting time (for definition see Section \ref{pre}). We consider a time-changed process $\{N^\phi(t)\}_{t\ge0}$ defined as follows:
\begin{equation}\label{tcbdp}
	N^\phi(t)\coloneq N(E^\phi(t)),\ t\ge0,
\end{equation}
where the inverse subordinator $\{E^\phi(t)\}_{t\ge0}$ is independent of the BDP $\{N(t)\}_{t\ge0}$.

As mentioned in Remark \ref{re2.2}, the process defined in (\ref{tcbdp}) reduces to the FBDP when $E^\phi(t)$ is an inverse stable subordinator. The time-changed Markov processes also known as the semi-Markov processes are topic of study for last few years (see Kurtz (1971), Kaspi and Maisonneuve (1988), Orsingher \textit{et al.} (2018), and references therein).  
 \begin{proposition}
 	The state probabilities $p^{\phi}(n,t)=\mathrm{Pr}\{N^\phi(t)=n\}$, $n\ge0$ solve the following system of fractional differential equations:
 	\begin{equation}
 	\mathscr{D}_tp^\phi(n,t)-p(n,0)\bar{\nu}(t)=-(\lambda_n+\mu_n)p^\phi(n,t)+\lambda_{n-1}p^\phi(n-1,t)+\mu_{n+1}p^\phi(n+1,t),\ n\ge0,
 	\end{equation}
 	with initial condition $p^\phi(1,0)=1$. Here, $\bar{\nu}(t)=\nu(t,\infty)$ and the operator $\mathscr{D}_t$ is the generalized fractional derivative as defined in (\ref{gfrder}).
 \end{proposition}
 \begin{proof}
 	From (\ref{tcbdp}), we have
 	\begin{equation}\label{lp1}
 		p^\phi(n,t)=\int_{0}^{\infty}p(n,x)h(x,t)\,\mathrm{d}x,\ n\ge0,
 	\end{equation} 
 	where $p(n,t)$ is the distribution of BDP and $h(x,t)$ is the density of $E^\phi(t)$.
 	
 	On applying the operator $\mathscr{D}_t$ on both sides of (\ref{lp1}) and using (\ref{invsubeq}), we get
 	\begin{equation*}
 		\mathscr{D}_tp^\phi(n,t)=-\int_{0}^{\infty}p(n,x)\mathcal{D}_xh(x,t)\,\mathrm{d}x=p(n,0)h(0,t)+\int_{0}^{\infty}h(x,t)\mathcal{D}_xp(n,x)\,\mathrm{d}x,\ n\ge0.
 	\end{equation*}
 	Finally, the result follows from on using (\ref{bdpequ}) and (\ref{lp1}).
 \end{proof}
\subsection{Extinction times of time-changed linear BDP}
 Let us consider a time-changed linear BDP $\{\tilde{N}_\phi(t)\}_{t\ge0}$ defined by 
\begin{equation}\label{tclbdp}
	\tilde{N}_\phi(t)\coloneqq\tilde{N}(E^\phi(t)),
\end{equation} 
where the inverse subordinator $\{E^\phi(t)\}_{t\ge0}$ is independent of the linear BDP $\{\tilde{N}(t)\}_{t\ge0}$. 

 Let $\tilde{\mathscr{T}}$ and $\tilde{\mathscr{T}}^\phi$ be the extinction times of linear BDP and its time-changed variants, respectively. That is,
$\tilde{\mathscr{T}}=\inf\{t>0:\tilde{N}(t)=0\}$ and $\tilde{\mathscr{T}}^\phi=\inf\{t>0:\tilde{N}_\phi(t)=0\}$. 
Further, we denote their distribution functions by $F(t)=\mathrm{Pr}\{\tilde{\mathscr{T}}\leq t\}$ and $F^\phi(t)=\mathrm{Pr}\{\tilde{\mathscr{T}}^\phi\leq t\}$, respectively. As immigration does not occur in the linear BDP, we have 
$F(t)=\mathrm{Pr}\{\tilde{N}(t)=0\}$ and $F^\phi(t)=\mathrm{Pr}\{\tilde{N}_\phi(t)=0\}$
for $t\ge0$ with $F(0)=F^\phi(0)=0$.

Next, we provide an asymptotic result for the tail of distribution of the extinction time of linear FBDP when the rate of birth is smaller than the rate of death. 
\begin{theorem}
If $\phi(\cdot)$ is slowly varying at zero with index $0\leq\delta<1$ then for $\lambda<\mu$, the function $t\mapsto\mathrm{Pr}\{\tilde{\mathscr{T}}^\phi>t\}$ is regularly varying at infinity and
\begin{equation}
	\mathrm{Pr}\{\tilde{\mathscr{T}}^\phi>t\}\sim-\ln\bigg(1-\frac{\lambda}{\mu}\bigg)\frac{\phi(1/t)}{\lambda\Gamma(1-\delta)}\ \ \text{as}\ t\to\infty.
\end{equation}	
\end{theorem}
\begin{proof}
	Let us consider the open set $(0,\infty)$ of the real line. Then, the extinction time $\tilde{\mathscr{T}}$ and $\tilde{\mathscr{T}}^\phi$ are the first exit times of $\{\tilde{N}(t)\}_{t\ge0}$ and $\{\tilde{N}_\phi(t)\}_{t\ge0}$ from $(0,\infty)$ that are defined as follows:
	\begin{equation*}
		\tilde{\mathscr{T}}=\inf\{t>0:N(t)\notin(0,\infty)\}\ \ \text{and}\ \ \tilde{\mathscr{T}}^\phi=\inf\{t>0:N^\phi(t)\notin(0,\infty)\}.
	\end{equation*}
	For $\lambda<\mu$, the extinction probability of the linear BDP is given by (see Bailey (1964), pp. 94)
	\begin{equation}\label{extprob}
		\mathrm{Pr}\{N(t)=0\}=
			\frac{\mu-\mu e^{-t(\mu-\lambda)}}{\mu-\lambda e^{-t(\mu-\lambda)}},\ \lambda<\mu.
	\end{equation}
	Thus, for $\lambda<\mu$, we have
	\begin{align*}
		\mathbb{E}\tilde{\mathscr{T}}=\int_{0}^{\infty}\mathrm{Pr}\{\tilde{\mathscr{T}}>t\}\,\mathrm{d}t&=\int_{0}^{\infty}\bigg(1-\frac{\mu-\mu e^{-t(\mu-\lambda)}}{\mu-\lambda e^{-t(\mu-\lambda)}}\bigg)\,\mathrm{d}t\\
		&=\int_{0}^{\infty}\frac{(\mu-\lambda)e^{-t(\mu-\lambda)}}{\mu-\lambda e^{-t(\mu-\lambda)}}\,\mathrm{d}t\\
		&=\int_{0}^{1}\frac{\mathrm{d}u}{\mu-\lambda u}=-\frac{1}{\lambda}\ln\bigg(1-\frac{\lambda}{\mu}\bigg)<\infty.
	\end{align*}
	The proof follows using the Corollary 2.3 of Ascione \textit{et al.} (2020).
\end{proof}
\begin{remark}
	If $\phi(\eta)=\eta^\alpha$, $\alpha\in(0,1)$ which slowly varying at zero with index $\alpha$, then for $\lambda<\mu$, in view of Remark \ref{re2.2}, the following result holds for the tail distribution of extinction time $\tilde{\mathscr{T}}^\alpha$ of the linear FBDP:
\begin{equation}\label{flbdpeta}		\mathrm{Pr}\{\tilde{\mathscr{T}}^\alpha>t\}\sim-\ln\bigg(1-\frac{\lambda}{\mu}\bigg)\frac{t^{-\alpha}}{\lambda\Gamma(1-\alpha)}\ \ \text{as}\ t\to\infty.
\end{equation}

Alternatively, from (\ref{lsp0}), we have
\begin{equation*}
	\mathrm{Pr}\{\tilde{\mathscr{T}}^\alpha>t\}=1-\mathrm{Pr}\{\tilde{N}^\alpha(t)=0\}=\frac{(\mu-\lambda)}{\lambda}\sum_{l=1}^{\infty}\bigg(\frac{\lambda}{\mu}\bigg)^lE_{\alpha,1}(-t^\alpha l (\mu-\lambda)),\ \lambda<\mu.
\end{equation*}
By using the following approximation (see Scalas (2006), Ascione \textit{et al.} (2020), Eq. (2.42)), we get
\begin{equation*}
	E_{\alpha,1}(-ct^\alpha)\sim\frac{t^{-\alpha}}{c\Gamma(1-\alpha)},\ c>0,\,\alpha\in(0,1)\ \ \text{as}\ \ t\to\infty,
\end{equation*}
we get
\begin{equation*}
	\mathrm{Pr}\{\tilde{\mathscr{T}}^\alpha>t\}\sim\frac{1}{\lambda\Gamma(1-\alpha)}\sum_{l=1}^{\infty}\bigg(\frac{\lambda}{\mu}\bigg)^l\frac{1}{l}=\frac{-1}{\lambda\Gamma(1-\alpha)}\ln\bigg(1-\frac{\lambda}{\mu}\bigg)\ \ \text{as}\ \ t\to\infty,
\end{equation*}
which agrees with (\ref{flbdpeta}).
\end{remark}

Next, we study the asymptotic behaviour of the distribution function of extinction time at zero for the process defined in (\ref{tclbdp}). First, we proof the following lemma.
\begin{lemma}\label{llema}
	For $\lambda<\mu$, the distribution function $F(t)$ of the extinction time of linear BDP is regularly varying at zero with index one. 
\end{lemma}
\begin{proof}
	For $t>0$, we have $F(t)=\mathrm{Pr}\{N(0)=0\}$. By using (\ref{extprob}), we get
	\begin{align*}
		\lim_{t\to0}\frac{F(xt)}{F(t)}&=\lim_{t\to0}\frac{1- e^{-tx(\mu-\lambda)}}{\mu-\lambda e^{-tx(\mu-\lambda)}}\frac{\mu-\lambda e^{-t(\mu-\lambda)}}{1-e^{-t(\mu-\lambda)}}\\
		&=\lim_{t\to0}\frac{1- e^{-tx(\mu-\lambda)}}{1- e^{-t(\mu-\lambda)}}=x,\ x>0.
	\end{align*}
	This completes the proof.
\end{proof}
\begin{theorem}\label{thm5.2}
	Let $\phi(\cdot)$ be regularly varying at infinity with index $\delta>0$. Then, for $\lambda<\mu$, the distribution function $F^\phi(t)$ is regularly varying at zero with index $\delta$ and
	\begin{equation}
		F^\phi(t)\sim\frac{1}{\Gamma(1+\delta)}\frac{\mu-\mu \exp{\left(\frac{\lambda-\mu}{\phi(1/t)}\right)}}{\mu-\lambda \exp{\left(\frac{\lambda-\mu}{\phi(1/t)}\right)}}\ \ \text{as}\ \ t\to0.
	\end{equation}
\end{theorem}
\begin{proof}
	Let us consider the Laplace transform of $F(t)$ and $F^\phi(t)$ as follows:
	\begin{equation}\label{lapdist}
		\mathcal{L}(F(t);w)=\int_{0}^{\infty}e^{-wt}\,\mathrm{d}F(t),\ w>0
	\end{equation}
	and 
	\begin{equation}\label{laptcdist}
		\mathcal{L}(F^\phi(t);w)=\int_{0}^{\infty}e^{-wt}\,\mathrm{d}F^\phi(t),\ w>0.
	\end{equation}
	In view of Lemma \ref{llema}, and using Tauberian theorem (see Feller (1971), p. 445, Theorem 2 and Theorem 3 of Section XIII.5), we have that the function (\ref{lapdist}) is regularly varying at infinity with index one and
	\begin{equation}\label{tblim}
		\mathcal{L}(F(t);w)\sim F(1/w)\ \ \text{as}\ \ w\to\infty.
	\end{equation}
	 From the definition of $E^\phi(t)$, it follows that $S^\phi(E^\phi(t))=t$ almost surely. So,
	$\mathrm{Pr}\{\tilde{\mathscr{T}}^\phi>t\}=\mathrm{Pr}\{\tilde{\mathscr{T}}>E^\phi(t)\}$ which implies $\mathrm{Pr}\{\tilde{\mathscr{T}}^\phi>t\}=\mathrm{Pr}\{S^\phi(\tilde{\mathscr{T}})>t\}$. From (\ref{laptcdist}), we have
	\begin{align*}
		\mathcal{L}(F^\phi(t);w)&=\int_{0}^{\infty}\int_{0}^{\infty}e^{-wt}\mathrm{Pr}\{S^\phi(x)\in\mathrm{d}t\}\mathrm{Pr}\{\tilde{\mathscr{T}}\in\mathrm{d}x\}\\
		&=\int_{0}^{\infty}e^{-x\phi(w)}\mathrm{Pr}\{\tilde{\mathscr{T}}\in\mathrm{d}x\}=\mathcal{L}(F(t);\phi(w)).
	\end{align*}
	Note $\phi(\cdot)$ and $\mathcal{L}(F(t);w)$ are regularly varying at infinity with index $\delta>0$ and  one, respectively. By using Proposition 1.5.7 of Bingham \textit{et al.} (1989), the function $\mathcal{L}(F^\phi(t);w)=\mathcal{L}(F(t);\phi(w))$ is regularly varying at infinity with index $\delta>0$. Thus, from (\ref{tblim}), we have
	\begin{equation*}
		\mathcal{L}(F^\phi(t);w)\sim F(1/\phi(w))\ \ \text{as}\ \ w\to\infty.
	\end{equation*}
	On using Tauberian theorem, the function $F^\phi(t)$ is regularly varying at zero with index $\delta$ and
	\begin{equation*}
		F^\phi(t)\sim\frac{1}{\Gamma(1+\delta)}F\bigg(\frac{1}{\phi(1/t)}\bigg)\ \ \text{as}\ \ t\to0.
	\end{equation*} 
	Hence, the proof follows on using (\ref{extprob}).
\end{proof}
\begin{remark}
	In the case of $\alpha$-stable subordinator, the Laplace exponent $\phi(\eta)=\eta^\alpha$, $\eta>0$, $\alpha\in(0,1)$ is regularly varying at infinity with index $\alpha$. Then, from Theorem \ref{thm5.2}, as $t\to0$, the distribution function of the extinction time for linear FBDP is regularly varying at zero with index $\alpha$. It can be approximated as follows: 
	\begin{equation}
		\mathrm{Pr}\{\tilde{\mathscr{T}}^\alpha\leq t\}\sim\frac{1}{\Gamma(1+\alpha)}\frac{\mu-\mu e^{-t^{-\alpha}(\mu-\lambda)}}{\mu-\lambda e^{-t^{-\alpha}(\mu-\lambda)}}\ \ \text{as}\ \ t\to0.
	\end{equation}
\end{remark}

\end{document}